\documentclass[11pt]{article}

\usepackage{amsmath}	
\usepackage{amssymb}
\usepackage{amsthm}
\usepackage{float}
\usepackage{color}
\usepackage{tikz,pgf}

\newtheorem{thm}{Theorem}
\newtheorem{lem}{Lemma}
\newtheorem{prop}{Proposition}
\newtheorem{cor}{Corollary}
\newtheorem{exm}{Example}
\newtheorem{dfn}{Definition}
\newtheorem{remark}{Remark}

\newcommand{\footremember}[2]{%
    \footnote{#2}
    \newcounter{#1}
    \setcounter{#1}{\value{footnote}}%
}
\newcommand{\footrecall}[1]{%
    \footnotemark[\value{#1}]%
} 

\usepackage{graphicx}	
\begin{document}

\begin{center}
{\Large 
 A generalization of rotation of binary sequences
 and its applications to  toggle dynamical systems
}

 \bigskip
 Erika Hanaoka\footremember{Tsuda}{Department of Computer Science, Tsuda University, Tsuda-machi 2-1-1, Kodaira, Japan} 
 and Taizo Sadahiro\footrecall{Tsuda}
\end{center}

\begin{abstract}
This paper generalizes the results of Joseph
and Roby\cite{JR} on a toggle dynamical system 
whose state space consists of independent
sets on a path graph.
Along the proof,
a simple generalization of
the rotation (or circular shift) of the binary sequences
arises.
We show each orbit of
this generalized rotation has a certain
statistical symmetry.
\end{abstract}

\section{Introduction}
The idea of a toggle group (a group generated by products of simple involutions) was introduced by Cameron and Fon-der-Flaass \cite{cameron1995orbits} 
to analyze certain actions on order ideals of posets, and has been generalized to many other contexts, see e.g., Striker \cite{striker2016rowmotion}. 
The {\em toggle} map is defined as follows:
\begin{dfn}
 Let $E$ be a finite set and ${\mathcal L}$ be a family of subsets of $E$.
 For each  $e\in E$, the toggle $\tau_e:{\mathcal L}\to{\mathcal L}$ is defined by
 \[
 \tau_e(X)= \begin{cases}
	   X\cup \{e\} & \mbox{ if } e\not\in X \mbox{ and } X\cup\{e\}\in {\mathcal L}\\
	   X\backslash \{e\} & \mbox{ if } e\in X \mbox{ and } X\backslash\{e\}\in {\mathcal L}\\
	   X & \mbox{ otherwise.} 
	 \end{cases}
 \]
\end{dfn}
The group generated by $\{\tau_e\,|\, e\in E\}$ is called the {\em toggle group}.
Many important actions on combinatorial objects have been shown
to have interpretations through the toggle groups\cite{striker2012promotion}.

Repeatedly applying a fixed element of the toggle group eventually cycles as each toggle is a bijection.
The {\em path graph} of size $N$ is the undirected graph equipped with the vertex set
$\{0,1,2,\ldots,N-1\}$ and the edge set $\{\{0,1\},\{1,2\},\cdots,\{N-2,N-1\}\}$.
Propp conjectured that around each orbit of independent sets on the path graph
under the action of $\varphi=\tau_{N-1}\circ\cdots\circ \tau_1\circ\tau_0$,
the number of times vertex $i$ occurs is same as vertex $N-i-1$. 
Joseph and Roby proved this conjecture (and more) in \cite{JR} using the notion of "snakes" introduced by Haddadan \cite{haddadan2014some}.

In this paper, we  prove a generalization (Theorem $\ref{thm:thirdmain}$), where the independent sets are replaced with
{\em more} independent sets, that is, a family of subsets of the vertices $\{0,1,\ldots,N-1\}$ 
not containing any pair of vertices whose mutual distance
is less than or equal to an integer $m$.
Therefore the ordinary independent sets studied by Joseph and Roby\cite{JR} is the case with $m=1$.
We prove this generalization by reducing the problem to a property of a generalized rotation of bit-strings (Theorem $\ref{thm:main}$).


\section{Generalized rotation}
\subsection{Generalized rotation}
Throughout this paper, $m,n$ denote  positive integers with $m\leq n$.
For a word $w$ over a finite alphabet,
$w_i$ denotes  the $i$-th letter of $w$ and $w_{[i,j]}$ denotes
the subword of the form $w_iw_{i+1}\cdots w_{j}$.
The length $n$ of the word $w$ is denoted by $|w|$.
For a finite alphabet $A$,
the set of finite words over $A$ is denoted by $A^*$.
The {\em generalized rotation} $\rho: \{0,1\}^n \rightarrow \{0,1\}^n$ 
is defined as follows: Let $w=w_0w_1w_2\cdots w_{n-1}\in\{0,1\}^n$
be a word over $\{0,1\}$ of length $n$.
Then, 
\[
 \rho(w) = \begin{cases}
	    w_{k+1}w_{k+2}\cdots w_{n-1}0\stackrel{k}{\overbrace{1\cdots1}} & \mbox{ if there exists } k<m \mbox{ s.t. } w_{[0,k]}=\stackrel{k}{\overbrace{1\cdots1}}0\\
	    w_mw_{m+1}\cdots w_{n-1}\stackrel{m}{\overbrace{1\cdots1}} & \mbox{ otherwise. }
	   \end{cases}
\]
It is clear that $\rho$ is a bijection and $\rho$ is the ordinary rotation for $m=1$.
Throughout the paper,  $p$ denotes  the smallest integer such that
$\rho^p(w)=w$,  that is, $p=\#\{\rho^k(w)|k\in{\mathbb Z}\}$.
We call the sequence $w, \rho^1(w), \rho^2(w),\ldots,\rho^{p-1}(w)$ the
{\em $\rho$-orbit} of $w$.

\begin{exm}
\label{exm:first}
When $m=3$, we have
\[
 \rho(w) =
 \begin{cases}
  w_1w_2\cdots w_n0 & w_0 = 0,\\
  w_2w_3\cdots w_n01 & w_0w_1 = 10,\\
  w_3w_4\cdots w_n011 & w_0w_1w_2 = 110,\\
  w_3w_4\cdots w_n111 & w_0w_1w_2 = 111.
 \end{cases}
\]
$w=1011110$ has the $\rho$-orbit of $w$ of length $9:$
\[
 \begin{array}{c|c}
  k & \rho^k(w)\\
  \hline
0 & 1    0    1    1    1    1    0\\
1 & 1    1    1    1    0    0    1\\
2 & 1    0    0    1    1    1    1\\
3 & 0    1    1    1    1    0    1\\
4 & 1    1    1    1    0    1    0\\
5 & 1    0    1    0    1    1    1\\
6 & 1    0    1    1    1    0    1\\
7 & 1    1    1    0    1    0    1\\
8 & 0    1    0    1    1    1    1
 \end{array}			   
\]\qed
\end{exm}

\subsection{Statistical symmetry of the orbits}
One of our aims in this paper is to show a statistical property
of the cumulative sum 
\[
 \rho^k(w)_0 + \rho^k(w)_1 + \cdots + \rho^k(w)_{j-1},
\]
for $j=1,2,\ldots,n$ where $\rho^k(w)_i$ stands for the $i$-th letter (or bit)
 of the word $\rho^k(w)$.
Here we recall the notation and definition of the multisets
(see, e.g., \cite{StanleyEC1}).
A {\em multiset} is intuitively a set with repeated elements.
More precisely, a {\em multiset} $M$ on a set $S$ is a pair $(S,\nu)$,
where $\nu:S\to {\mathbb N}$ is the {\em multiplicity} function.
For example $M=\{a,b,b,c,c,c\}$ is a multiset on $S=\{a,b,c\}$ with
the multiplicity $\nu(a)=1, \nu(b)=2, \nu(c)=3$.
Let $a(w)$ be the number of the digits $1$ in $w$.
We define  $L^{(j)}$ as the multiset on $\{0,1,\ldots,a(w)\}$
consisting of the left cumulative sums:
\[
 L^{(j)}
 =
\left\{
   \sum_{i=0}^{j-1}\rho^k(w)_i 
   \,\middle|\,
   k=0,1,\ldots,p-1
  \right\},
\]  
for $j=1,2,\ldots,n$.
We also define $R^{(j)}$ as the multiset consisting of the right cumulative sums:
\[
R^{(j)}
=
  \left\{
   \sum_{i=0}^{j-1}\rho^k(w)_{n-i-1} 
   \,\middle|\,
   k=0,1,\ldots,p-1
   \right\}.
\]
For convenience, we define
$L^{(0)}=R^{(0)}=\big\{\stackrel{p}{\overbrace{0,0,\cdots, 0}}\big\}$,
which consists of only $0$s.
We also define
\[
 \nu_{L^{(j)}}\left(s\right)
 =
 \#\left\{k\in\{0,1,\ldots,p-1\}\,\middle|\, \sum_{i=0}^{j-1}\rho^k(w)_i = s\right\},
\]
and
\[
 \nu_{R^{(j)}}\left(s\right)
 =
 \#\left\{k\in\{0,1,\ldots,p-1\}\,\middle|\, \sum_{i=0}^{j-1}\rho^k(w)_{n-i-1}  = s\right\}.
\]
\begin{thm}
\label{thm:main}
 Let $w\in\{0,1\}^n$ be a word over $\{0,1\}$ of length $n$.
 Then, for $j=0,1,\ldots,n-1$,
 as multisets
 \begin{equation}
  \label{eq:mainth}
   L^{(j)} = R^{(j)},
 \end{equation}
 i.e., $\nu_{L^{(j)}}(s)=\nu_{R^{(j)}}(s)$
 for $s\in\{0,1,\ldots,a(w)\}$.
\end{thm}
As an immediate corollary of this theorem, we have the following.
\begin{cor}
 \begin{equation}
  \label{eq:cor}
   \sum_{k=0}^{p-1}\rho^k(w)_j
   =
   \sum_{k=0}^{p-1}\rho^k(w)_{n-1-j}.
 \end{equation}
\end{cor}

\begin{exm} Let $m=3$ and
 $w=1011110\in\{0,1\}^7$, same as in Example $\ref{exm:first}$.
Then we have the following tables of left cumulative sums $L^{(j)}$ and
right cumulative sums $R^{(j)}$.
\begin{table}[H]
\begin{center}
\begin{tabular}{r|rrrrrrrrr}
$k\backslash j$ & 0 & 1 & 2 & 3 & 4 & 5 & 6 & 7 \\ 
  \hline
0  & 0 & 1 & 1 & 2 & 3 & 4 & 5 & 5  \\   
1  & 0 & 1 & 2 & 3 & 4 & 4 & 4 & 5  \\   
2  & 0 & 1 & 1 & 1 & 2 & 3 & 4 & 5  \\ 
3  & 0 & 0 & 1 & 2 & 3 & 4 & 4 & 5  \\ 
4  & 0 & 1 & 2 & 3 & 4 & 4 & 5 & 5  \\ 
5  & 0 & 1 & 1 & 2 & 2 & 3 & 4 & 5  \\ 
6  & 0 & 1 & 1 & 2 & 3 & 4 & 4 & 5  \\ 
7  & 0 & 1 & 2 & 3 & 3 & 4 & 4 & 5  \\ 
8  & 0 & 0 & 1 & 1 & 2 & 3 & 4 & 5  \\
   \hline
\end{tabular}
~~~
\begin{tabular}{r|rrrrrrrr}
$k\backslash j$ & 0 & 1 & 2 & 3 & 4 & 5 & 6 & 7 \\ 
  \hline
0 & 0 & 0 & 1 & 2 & 3 & 4 & 4 & 5 \\
1 & 0 & 1 & 2 & 3 & 4 & 4 & 5 & 5 \\
2 & 0 & 1 & 1 & 2 & 2 & 3 & 4 & 5 \\
3 & 0 & 1 & 1 & 2 & 3 & 4 & 4 & 5 \\
4 & 0 & 1 & 2 & 3 & 3 & 4 & 4 & 5 \\
5 & 0 & 0 & 1 & 1 & 2 & 3 & 4 & 5 \\
6 & 0 & 1 & 1 & 2 & 3 & 4 & 5 & 5 \\
7 & 0 & 1 & 2 & 3 & 4 & 4 & 4 & 5 \\
8 & 0 & 1 & 1 & 1 & 2 & 3 & 4 & 5 \\
   \hline
\end{tabular}
\end{center}
\caption{Left: table of left cumulative sums $\sum_{i=0}^{j-1}\rho^k(w)_i$ for $m=3$,
 Right: table of right cumulative sums $\sum_{i=0}^{j-1}\rho^k(w)_{n-i-1}$ for $m=3$.}
\end{table}
We can summarize these tables by the {\em frequency tables},
that is, tables whose $j$-th column vector is
\[
\left(\nu_{L^{(j)}}(0), \nu_{L^{(j)}}(1), \ldots, \nu_{L^{(j)}}(5)\right)
=
\left(\nu_{R^{(j)}}(0), \nu_{R^{(j)}}(1), \ldots, \nu_{R^{(j)}}(5)\right).
\]


\begin{figure}[H]
\begin{center}
\begin{tabular}{r|rrrrrrrr}
 $k~\backslash~ j$ & 0 &  1 & 2 & 3 & 4 & 5 & 6 & 7 \\ 
  \hline
 0 & 9 &   2 &   0 &   0 &   0 &   0 &   0 & 0\\ 
 1 & 0 &   7 &   6 &   2 &   0 &   0 &   0 & 0\\ 
 2 & 0 &   0 &   3 &   4 &   3 &   0 &   0 & 0\\ 
 3 & 0 &   0 &   0 &   3 &   4 &   3 &   0 & 0\\ 
 4 & 0 &   0 &   0 &   0 &   2 &   6 &   7 & 0\\ 
 5 & 0 &   0 &   0 &   0 &   0 &   0 &   2 & 9\\ 
   \hline
\end{tabular}
~~
\caption{Frequency table of cumulative sums $\sum_{i=0}^{j-1}\rho^k(1011110)_i$ for $j=0,1,2,\ldots,7$,
i.e., table of $\nu_{L^{(j)}}(s)$ where $m=3$. 
}
\label{fig:freqtab}
\end{center}
\end{figure}

\end{exm}

Since the number $a(w)$ of digits $1$ in $\rho^{k}(w)$ does not depend
on $k$, we have
\[
 \rho^k(w)_0 + \rho^k(w)_1 + \cdots + \rho^k(w)_{j-1} = s
 ~~\Longleftrightarrow~~
 \rho^k(w)_{n-1} + \rho^k(w)_{n-2} + \cdots + \rho^k(w)_{j} = a(w)-s.
\]
Therefore, we have
\begin{equation}
 \label{eq:cumsumsym}
 \nu_{{L}^{(j)}}(s) = \nu_{R^{(n-j)}}(a(w)-s) = \nu_{L^{(n-j)}}(a(w)-s).
\end{equation}
By theorem $\ref{thm:main}$ and the relation $(\ref{eq:cumsumsym})$, if  $n$ is even, then
we have
\[
 \nu_{{L}^{(\frac{n}{2})}}(s) = \nu_{L^{(\frac{n}{2})}}(a(w)-s), 
\]
where $a(w)$ is the number of digits $1$ in $w$.

\begin{remark}
The reverse ${\rm Rev}(w)$ of a word $w=w_1w_2\cdots w_n\in \{0,1\}^n$ is defined by
${\rm Rev}(w) = w_nw_{n-1}\cdots w_2w_1$.
Then, it is clear that
\begin{equation}\label{eq:fund}
 \rho^{-1} = {\rm Rev}\circ\rho\circ{\rm Rev}.
\end{equation}
We remark that Theorem $1$ would be easily shown if it were true 
that the $\rho$-orbit of $w$ contains its reverse ${\rm Rev}(w)$:
Let $j$ be the smallest positive integer such that ${\rm Rev}(w)=\rho^j(w)$
and let ${\mathcal O}=\{\rho^k(w)\,|\,k\in{\mathbb Z}\}$. Then it is clear that
$
{\mathcal O}
 = \{\rho^{-k}(w)\,|\,k\in {\mathbb Z}\}
$
and
$
{\mathcal O}
 = \{\rho^{j+k}(w)\,|\,k\in {\mathbb Z}\}
 = \{\rho^{k}\circ {\rm Rev}(w)\,|\,k\in {\mathbb Z}\}
$.
From $(\ref{eq:fund})$ and the fact ${\rm Rev}^2$ is the identity map,
we have
$
 \rho^{-k} = {\rm Rev}\circ\rho^k\circ {\rm Rev}.
$
Let ${\rm Rev}({\mathcal O})$ the set $\{{\rm Rev}(w)\,|\,w\in{\mathcal O}\}$.
Then,  we have
\begin{eqnarray*}
{\rm Rev}({\mathcal O})
 & = & 
 \{{\rm Rev}\circ\rho^{k}(w)\,|\,k\in {\mathbb Z}\}
 =
 \{{\rm Rev}\circ\rho^{k}\circ{{\rm Rev}}(w)\,|\,k\in {\mathbb Z}\}\\
 & =& 
 \{\rho^{-k}(w)\,|\,k\in {\mathbb Z}\}
 =
 \{\rho^{k}(w)\,|\,k\in {\mathbb Z}\}
 =
 {\mathcal O},
\end{eqnarray*}
which implies Theorem $1$.
However, there are words $w$
whose $\rho$-orbit does not contain ${\rm Rev}(w)$.
For example, when $m=2$, the $\rho$-orbit of
$w=00100101$ is 
\[
 00100101,~
 01001010,~
 10010100,~
 01010001,~
 10100010,~
 10001001,
\]
which does not contain ${\rm Rev}(w)$.
\end{remark}

Let $w=w_0w_1\cdots w_{n-1}\in\{0,1\}^n$ be a word of length $n$.
Then, the subword $w_{i}w_{i+1}\cdots w_{j}$ of $w$ is denoted by
$w_{[i,j]}$.
We define an extension sequence of words $w^{(0)}, w^{(1)}, w^{(2)},\ldots$ as follows.
We define $w^{(0)}=w$, and for $j\geq 0$,  $w^{(j+1)}$ is obtained as an
extension of $w^{(j)}$ defined by
\[
 w^{(j+1)} =
\begin{cases}
 w^{(j)}0\stackrel{k}{\overbrace{1\cdots1}} &   
 \mbox{ if there exists } k < m \mbox{ s.t. } \rho^j(w)_{[0,k]}=\stackrel{k}{\overbrace{11\cdots1}}0\\
 w^{(j)}\stackrel{m}{\overbrace{1\cdots1}} & \mbox{ otherwise. }
\end{cases}
\]
Therefore $w^{(k)}$ contains $\rho^k(w)$ as its suffix of final $n$ bits.
Let $p$ be the size of $\rho$-orbit of $w$.
Then, $p$ is the smallest non-negative integer such that
$w$ is  a  suffix of $w^{(p)}$.
Let $n+l$ be the length of the word $w^{(p)}$.
Then, we define 
$\overline{w} = \overline{w}_0\overline{w}_1\cdots\overline{w}_{l-1}$ to
be the word obtained by removing the suffix $w$ from $w^{(p)}$.
Thus, we can see $\overline{w}$ as a compact representation of the $\rho$-orbit
of $w$.
The indices of $\overline{w}$ is always considered to be in
$\{0,1,\ldots,l-1\}$ by taking modulo $l$.
Let the sequence ${\mathbf c}_w=(c_0, c_1, \ldots, c_{p})$
be defined by 
\begin{equation}\label{eq:sequence}
 c_k = |w^{(k)}| - n, 
\end{equation}
where $|w^{(k)}|$ is the length of the word $w^{(k)}$.
In other words, ${\mathbf c}_w=(c_0, c_1, \ldots, c_{p})$
is the rising subsequence of $(0,1,\ldots,l)$ which satisfies
\begin{equation}
 \label{eq:orbitcontained}
  \rho^k(w) = w^{(p)}_{[c_k, c_k+n-1]},
\end{equation}
and therefore we have the following lemma:
\begin{lem}
 \label{lem:lemfirst}
\begin{equation}
 \label{eq:orbprefix}
 w^{(p)}_{[c_k, c_{k+1}-1]}=
  \begin{cases}
   \stackrel{m}{\overbrace{11\cdots 1}} & c_{k+1}-c_k = m \mbox{ and } \rho^k(w)_{m-1}=1,\\
   \stackrel{c_{k+1}-c_k-1}{\overbrace{11\cdots 1}}0 & \mbox{ otherwise}.
  \end{cases}
\end{equation}
\end{lem}
We define another word $\widehat{w}=\widehat{w}_0\widehat{w}_1\cdots \widehat{w}_{l-1}$ 
by removing the prefix $w$ of starting $n$ bits from $w^{(p)}$.
Then, we have
\[
 \widehat{w}_{[c_{k}-n,c_k-1]} = w^{(p)}_{[c_{k},c_k+n-1]} = \rho^k_m(w).
\]
\begin{lem}
\label{lem:relbarhat}
\[
\widehat{w}_{[c_k, c_{k+1}-1]}
 =
 {\rm Rev}\left(
 \overline{w}_{[c_k, c_{k+1}-1]}
 \right)
 =
 \begin{cases}
  \stackrel{m}{\overbrace{11\cdots 1}} & \mbox{ if } c_{k+1}-c_k = m \mbox{ and } \rho^k(w)_{m-1}=1,\\
  \stackrel{c_{k+1}-c_k-1}{0\overbrace{11\cdots 1}} & \mbox{ otherwise}.
 \end{cases}
\]
\end{lem}
\begin{proof}
By $(\ref{eq:orbprefix})$ and
\[
\widehat{w}_{[c_k, c_{k+1}-1]}
=
 w^{(p)}_{[c_k+n, c_{k+1}-1+n]}
 =
 {\rm Rev}\left(
 w^{(p)}_{[c_k, c_{k+1}-1]}
 \right)
 =
 {\rm Rev}\left(
 \overline{w}_{[c_k, c_{k+1}-1]}
 \right),
\]
where ${\rm Rev}(w)$ denotes the reverse of $w$.
\end{proof}

\bigskip
The main idea of the proof of Theorem $\ref{thm:main}$ can be informally stated as follows.
An element of the left hand side of $(\ref{eq:mainth})$ can be expressed in terms of
$\overline{w}$:
\[
\sum_{i=0}^{j-1}\rho^k(w)_i  =
\sum_{i=0}^{j-1}\overline{w}_{c_k+i},
\]
for $j=1,2,\ldots,n$.
A similar expression of elements of the  right hand side
of $(\ref{eq:mainth})$ can be obtained by using
$\rho^{-1}$ and $\widehat{w}$ instead of $\rho$ and $\overline{w}$.
These expressions are used to prove the
equality of these two multisets.
The equality can be easily shown for
$j < m$, and induction on $j$ is used
for $j\geq m$.

\begin{exm}
 \label{exm:barw}
 When $m=3$ and $w = 1011110$, we have\\
\[
 \begin{array}{ll}
  w^{(0)}  =  \underline{1011110} &                     c_0 =  0\\ 
  w^{(1)}  =  10\underline{1111001} &			c_1 =  2\\ 
  w^{(2)}  =  10111\underline{1001111} &		c_2 =  5\\ 
  w^{(3)}  =  1011110\underline{0111101} &		c_3 =  7\\ 
  w^{(4)}  =  10111100\underline{1111010} &		c_4 =  8\\ 
  w^{(5)}  =  10111100111\underline{1010111} &		c_5 =  11\\
  w^{(6)}  =  1011110011110\underline{1011101} &	c_6 =  13\\
  w^{(7)}  =  101111001111010\underline{1110101} &	c_7 =  15\\
  w^{(8)}  =  101111001111010111\underline{0101111} &	c_8 =  18\\
  w^{(9)}  =  1011110011110101110\underline{1011110} &	c_9 =  19\\
 \end{array}
\]
Underlined part of $w^{(k)}$ is equal to $\rho^k(w)$.
Therefore, removing the suffix $w$ of final $n$ bits from $w^{(9)}$, we have 
\[
 \overline{w} = 1011110011110101110,
\]
and the length $l$ of $\overline{w}$ is $19$.
By removing the starting $n$ bits from $w^{(9)}$, we have 
\[
 \widehat{w} = 0111101011101011110.
\]
 \[
 \begin{array}{c|c|c|c|c|c|c|c|c|c|c|c|c|c|c|c|c|c|c|c|}
  & c_0 &  & c_1 &  &  & c_2 &  & c_3 & c_4 &  &  & c_5 &  & c_6 &  & c_7 &  &  & c_8 \\
  i & 0 & 1 & 2 & 3 & 4 & 5 & 6 & 7 & 8 & 9 & 10 & 11 & 12 & 13 & 14 & 15 & 16 & 17 & 18 \\
  \hline
  \overline{w}_i & {\color{red}1} & 0 & {\color{red}1} & {\color{red}1} & {\color{blue}1} & 
{\color{red}1} & 0 & {\color{black}0} & {\color{red}1} & {\color{red}1} & {\color{blue}1} & {\color{red}1} & 0 & {\color{red}1} & 0 & {\color{red}1} & {\color{red}1} & {\color{blue}1} & {\color{black}0}\\
  \hline
  \widehat{w}_i & 0 & {\color{red}1} & {\color{blue}1} & {\color{red}1} & {\color{red}1} & 0 & {\color{red}1} & {\color{black}0} & {\color{blue}1} & {\color{red}1} & {\color{red}1} & 0 & {\color{red}1} & 0 & {\color{red}1} & {\color{blue}1} & {\color{red}1} & {\color{red}1} & {\color{black}0}\\
 \end{array}
 \]

\qed
\end{exm}

\bigskip
In the following, 
$l$ denotes the length of the word $\overline{w}$, and
we regard the indices of the letters in the words $\overline{w}$ and $\widehat{w}$
are in $\{0,1,\ldots,l-1\}$  by taking modulo $l$.
We divide the set $I=\{0,1,\ldots,l-1\}$ of indices of $\overline{w}$
into two disjoint subsets, ${I}_0=\{i\,|\, \overline{w}_i=0\}$ and 
${I}_1=\{i\,|\, \overline{w}_i=1\}$. It is obvious that
\[
 I= {I}_0\cup {I}_1, \mbox{ and } {I}_0\cap {I}_1 = \emptyset.
\]
By $(\ref{eq:orbprefix})$, it is clear that ${I}_0 \subset \left\{c_0-1,c_1-1,\ldots, c_{p-1}-1\right\}$.
We define
\begin{eqnarray*}
 {I}_T & = & {I}_1\cap \left\{c_0-1,c_1-1,\ldots, c_{p-1}-1\right\}\\
 &  = & \left\{i\in I\,\middle|\, \overline{w}_{[i-m+1,~i]}=\stackrel{m}{\overbrace{11\dots 1}},\,
	i = c_j-1 \mbox{ for some }j\right\},
\end{eqnarray*}
and
$
  {I}_H={I}_1\backslash {I}_T,
$
where $i-m+1$ is considered to be in $I$ by taking modulo $l$.
Thus we have a decomposition,
$
 I = {I}_0 \cup {I}_H \cup {I}_T.
$
One of the most important properties of this decomposition is
\begin{equation}
 \label{eq:I0T}
 \{c_0, c_1, \ldots, c_{p-1}\}
 =
 \{k+1\,|\, k \in {I}_0\cup {I}_T\},
\end{equation}
from which we obtain another expression of $L^{(j)}$:
\[
L^{(j)} =
\left\{
\sum_{i=0}^{j-1}\overline{w}_{k+i+1} \,
\middle|\,
k\in {I}_0\cup {I}_T
\right\}.
\]

By using $\widehat{w}$ instead of $\overline{w}$,
we define another decomposition $I=\widehat{I}_0\cup \widehat{I}_T\cup \widehat{I}_H$
in the following way. Let $\widehat{I}_0 = \{i\,|\,\widehat{w}_i=0\}$ and
$\widehat{I}_1 = \{i\,|\,\widehat{w}_i=1\}$.
Then, by Lemma $\ref{lem:relbarhat}$, it is 
clear that $\widehat{I}_0\subset\{c_0, c_1, \ldots, c_{p-1}\}$.
We subdivide $\widehat{I}_1$ into two disjoint subsets:
\[
 \widehat{I}_T
 =
 \widehat{I}_1\cap \left\{c_0, c_1, \ldots, c_{p-1}\right\},
 ~~~
 \widehat{I}_H=\widehat{I}_1\backslash \widehat{I}_T.
\]
One of the most important properties of this decomposition is
\begin{equation}
 \label{eq:hatI0T}
 \{{c}_0, {c}_1, \ldots, {c}_{p-1}\}
 =
 \widehat{I}_0\cup \widehat{I}_T,
\end{equation}
and therefore we obtain another expression of $R^{(j)}$:
\[
R^{(j)} =
\left\{
\sum_{i=0}^{j-1}\widehat{w}_{k-i-1} \,
\middle|\,
k\in {\widehat{I}}_0\cup {\widehat{I}}_T
\right\}.
\]

\begin{exm}
\label{exm:handt}
Let $w=1011110$ and $m=3$. 
As we have shown in Example $\ref{exm:barw}$,
$(c_0,c_1,\ldots, c_9)=(0,2,5,7,8,11,13,15,18,19)$,
and we have
\[
 {I}_{0} = \{1,6,7,12,14,18\},~~ {I}_{H} = \{0,2,3,5,8,9,11,13,15,16\}, ~~
  {I}_{T} = \{4,10,17\},
\]
and
\[
 \widehat{I}_{0} = \left\{0,5,7,11,13,18\right\},~~
 \widehat{I}_{H} = \left\{1,3,4,6,9,10,12,14,16,17\right\},~~
 \widehat{I}_{T} =\left\{2,8,15\right\}.
\]
 \qed
\end{exm}

Let $j\leq n$ be a non-negative integer
and $a,b\in\{0,H,T\}$.
Then we define 
the multiset $M_{a,b}^{(j)}$ by
\[
 M^{(j)}_{a,b} =
 \left\{\sum_{i=0}^{j-1}\overline{w}_{k+i} \,\middle|\, k\in {I}_{a}, k+j-1\in {I}_{b}\right\}.
\]
The left hand side $L^{(j)}_m$ of $(\ref{eq:mainth})$ 
has the following decomposition:
\begin{equation}
 \label{eq:Ldeomposition}
 L^{(j)} = 
 \left(\bigcup_{b\in \{0,T,H\}}M_{0,b}^{(j+1)}\right)
 \cup
 \left(\bigcup_{b\in \{0,T,H\}}M_{T,b}^{(j+1)}-1\right),
\end{equation}
where $M-1$ denotes the multiset $\{m-1\,|\,m\in M\}$
for a multiset $M$ of integers.

Then we define 
the multiset $\widehat{M}_{a,b}^{(j)}$ by
\[
 \widehat{M}^{(j)}_{a,b} =
  \left\{
 \sum_{i=0}^{j-1}\widehat{w}_{k-i}
 \,\middle|\, 
 k\in \widehat{I}_{a}, k-j+1\in \widehat{I}_{b}
 \right\}.
\]
Then, the right hand side $R^{(j)}$ of $(\ref{eq:mainth})$ 
has the following decomposition:
\begin{equation}
 \label{eq:Rdecomposition}
 R^{(j)} = 
 \left(\bigcup_{b\in \{0,T,H\}}\widehat{M}_{0,b}^{(j+1)}\right)
 \cup
 \left(\bigcup_{b\in \{0,T,H\}}\widehat{M}_{T,b}^{(j+1)}-1\right).
\end{equation}
By $(\ref{eq:Ldeomposition})$ and $(\ref{eq:Rdecomposition})$,
to prove $L^{(j)}=R^{(j)}$,
it suffices to show
$M_{a,b}^{(j)} = \widehat{M}_{a,b}^{(j)}$
for all $a,b\in\{0,T,H\}$.

\begin{exm}
\label{exm:M}
Let $w=1011110$ and $m=3$, the same as the previous examples.
Table $\ref{tab:M}$ summarizes $M^{(3)}_{m,(a,b)}(w)$ for $w=1011110$ and $m=3$.
For instance, as we have seen in Example $\ref{exm:handt}$,
${I}_{0} = \{1,6,7,12,14,18\}$, and hence
$
 \left({I}_{0} + 2 \right)\cap {I}_{0} = \{14,1\}.
$
Therefore we have
$
 M^{(3)}_{0,0} =
 \left\{
 \overline{w}_{12} + \overline{w}_{13} + \overline{w}_{14} = 1,~~
 \overline{w}_{18} + \overline{w}_{0} + \overline{w}_{1} = 1\right\}.
$
Also
${I}_{T} = \{4,10,17\}, {I}_H=\{0,2,3,5,8,9,11,13,15,16\}$ and hence
$
 \left({I}_{T} + 2\right) \cap {I}_{H} = \{0\}.
$
Therefore, we have
$
 M^{(3)}_{{T},H}
 =
 \left\{
 \overline{w}_{17} + \overline{w}_{18} + \overline{w}_{0} = 2
 \right\}.
$
Table $\ref{tab:M}$ shows that
\begin{eqnarray*}
 L^{(2)} & = &
 \left(\bigcup_{b\in \{0,T,H\}}M_{0,b}^{(3)}\right)
 \cup
 \left(\bigcup_{b\in \{0,T,H\}}M_{T,b}^{(3)}-1\right)\\
 & = &
  \{1,1\} \cup \{1,2,2,2\} \cup \left(\{2,2\}-1\right)
  \cup
  \left(\{2\}-1\right)\\
 & = & \{1,1,1,1,1,1,2,2,2\}.
\end{eqnarray*}

\begin{table}[H]
\begin{center}
\begin{tabular}{|c|ccc|}
\hline
$a\backslash b $ & $0$ & $T$ & $H$ \\
\hline
 $0$ & $\{1,1\}$ & $\emptyset$ & $\{1,2,2,2\}$   \\ 
 $T$ & $\{2,2\}$ & $\emptyset$ & $\{2\}$         \\ 
 $H$ & $\{1,2\}$ & $\{3,3,3\}$ & $\{2,2,2,3,3\}$ \\ 
\hline
\end{tabular}
\caption{Table of ${M}^{(3)}_{a,b}$ for $w=1011110$ and $m=3$}
\label{tab:M}
\end{center}
\end{table}

\qed
\end{exm}

As can be seen in the examples above, the following two lemmas  relating $\overline{w}$ and $\widehat{w}$ hold.
\begin{lem}
\label{lem:slide}
Let $w = w_0w_1\cdots w_{n-1}\in\{0,1\}^n$ be a word of length $n$,
and $\overline{w}=\overline{w}_0\overline{w}_1\cdots\overline{w}_{l-1}$ 
and $\widehat{w}=\widehat{w}_0\widehat{w}_1\cdots\widehat{w}_{l-1}$be as defined above. Then
\begin{equation}
 \label{eq:barhatshift}
 \widehat{w}_i = \overline{w}_{i+n},
\end{equation}
where $i+n$ is considered to be in $\{0,1,\ldots,l-1\}$
by taking modulo $l$. In other words, 
$
 \rho^n(\overline{w})=\widehat{w}.
$
where $\rho$ is the ordinary rotation with $m=1$.
\end{lem}
\begin{proof}
This is clear from the definition of $\overline{w}$ and $\widehat{w}$.
\end{proof}

\begin{lem}
\label{lem:diff}
 The following maps are bijections.
\begin{eqnarray}
 {I}_0 \ni i  & \mapsto  & i - n  \in \widehat{I}_0.\label{eq:0}\\
 {I}_1 \ni i  & \mapsto  & i - n  \in \widehat{I}_1.\label{eq:1}\\
 I_{0}\cup I_{T} \ni i &  \mapsto &  i+1 \in \widehat{I}_{0}\cup\widehat{I}_{T}.\label{eq:0orT}\\
 {I}_{H} \ni i &  \mapsto &  i+1 \in \widehat{I}_{H}.\label{eq:H}\\
 {I}_{T} \ni i & \mapsto  &i - m +1 \in \widehat{I}_{T}\label{eq:T}.
\end{eqnarray}

\end{lem}
\begin{proof}
Since $\overline{w}$ is obtained from $w^{(p)}$ by
removing its suffix $w$, and
$\widehat{w}$ is obtained from the same sequence $w^{(p)}$
by removing its prefix $w$,
we see that $(\ref{eq:0})$ and $(\ref{eq:1})$ are bijections.
From $(\ref{eq:I0T})$ and $(\ref{eq:hatI0T})$,
it clearly follows that $(\ref{eq:0orT})$ is a bijection. This also shows that
$(\ref{eq:H})$ is a bijection 
since $I_H = I\backslash\left(I_0\cup I_T\right)$ and 
$\widehat{I}_H = I\backslash\left(\widehat{I}_0\cup \widehat{I}_T\right)$.

If $i\in I_T$ then we have $i=c_{k+1}-1$ for some $k$ which implies
$c_k = c_{k+1}-m \in \widehat{I}_T$ by Lemma $\ref{lem:relbarhat}$.
Conversely if $i \in \widehat{I}_T$ we have $i+m-1\in I_T$.
Therefore $(\ref{eq:T})$ is a bijection.
\end{proof}

Now we start to prove 
\begin{equation}
 \label{eq:elements}
 M_{a,b}^{(j)}=\widehat{M}_{a,b}^{(j)} 
\end{equation}
for $a,b\in\{0,T,H\}$ and $j=1,2,\ldots,n$.
We prove this by induction on $j$.
When $j=1$, $(\ref{eq:elements})$ is clear
since 
\begin{equation}
 \label{eq:j1}
M_{a,b}^{(1)}=\widehat{M}_{a,b}^{(1)}=
\begin{cases}
 \emptyset & a \neq b,\\
 \{\stackrel{|{I}_0|}{\overbrace{0,0,\ldots,0}}\} & a = b = 0,\\
 \{\stackrel{|{I}_a|}{\overbrace{1,1,\ldots,1}}\} & a = b \in \{H,T\}.\\
\end{cases}
\end{equation}
We prove the cases where $j=2,\ldots,m$ first,
and then prove for $j>m$ by using induction.
First we prove  Lemma $\ref{lem:00TT}$ and $\ref{lem:0orTto0orT}$
which hold for $j=1,2,\ldots,n$.
\begin{lem}
 \label{lem:00TT}
 For $j=1,2,\ldots,n$, 
 \[
 M^{(j)}_{0,0} = \widehat{M}^{(j)}_{0,0},
 ~~
 M^{(j)}_{T,T} = \widehat{M}^{(j)}_{T,T}.
 \]
\end{lem}
\begin{proof}
The first equation is clear from the fact that $\widehat{w}$
is obtained from $\overline{w}$ by applying (ordinary) rotations.
Suppose that $i\in I_T$ and $i+j-1\in I_T$.
Then, by Lemma $\ref{lem:diff}$, $i-m+1\in\widehat{I}_T$ and $i+j-m\in\widehat{I}_T$.
There exist some $r$ and $s$ such that $c_r=i+1$ and $c_s=i+j$,
and therefore $c_{r-1} = i-m+1 \in \widehat{I}_T$ and $c_{s-1}=i+j-m\in\widehat{I}_T$.
By Lemma $\ref{lem:relbarhat}$,  we have $\overline{w}_{c_k} + \cdots + \overline{w}_{c_{k+1}-1}
 = \widehat{w}_{c_k} + \cdots + \widehat{w}_{c_{k+1}-1}$
for every $k$. Therefore,
\begin{eqnarray*}
 \overline{w}_i + \overline{w}_{i+1} + \cdots + \overline{w}_{i+j-1}
  & = & \overline{w}_{c_r} + \overline{w}_{c_r + 1} + \cdots + \overline{w}_{c_{s-1}-1} + m +1\\ 
 & = &  \widehat{w}_{c_r} + \widehat{w}_{c_r + 1} + \cdots + \widehat{w}_{c_{s-1}-1} + m + 1\\
 & = &  \widehat{w}_{i-m+1} + \widehat{w}_{i-m+2} + \cdots + \widehat{w}_{i+j-m}.
\end{eqnarray*}
See Figure $\ref{fig:MTT}$.

\begin{figure}[H]
 \begin{center}
  \begin{tikzpicture}
   \draw (-1.2,0) node {$\overline{w}$};
   \draw (-1.2,-1) node {$\widehat{w}$};
   \draw (-1, 0) -- (7,0);
   \draw (0, 0.1) -- (0,-0.1);
   \draw (1.5, 0.1) -- (1.5,-0.1);
   \draw (6, 0.1) -- (6,-0.1);
   \draw (4.5, 0.1) -- (4.5,-0.1);
   \draw (1.3,0.25) node {$T$};
   \draw (5.8,0.25) node {$T$};

   \draw (2.5, 0.1) -- (2.5,-0.1);
   \draw (3.3, 0.1) -- (3.3,-0.1);

   \draw[line width=1, >=latex, ->] (2.5,0) -- (3.3,0);
   \draw[line width=1, >=latex, ->] (3.3,0) -- (4.5,0);
   \draw[line width=1, >=latex, ->,blue] (4.5,0) -- (6,0);
   \draw[line width=1, >=latex, ->] (1.5,0) -- (2.5,0);
   \draw[line width=1, >=latex, ->] (3.3,-1) -- (2.5,-1);
   \draw[line width=1, >=latex, <-] (1.5,-1) -- (2.5,-1);
   \draw[line width=1, >=latex, <-] (3.3,-1) -- (4.5,-1);
   \draw[line width=1, >=latex, <-,blue] (0,-1) -- (1.5,-1);
   \draw[dashed] (2.5, 0.) -- (2.5,-1);
   \draw[dashed] (3.3, 0.) -- (3.3,-1);

   \draw[dashed] (0,0) -- (0,-1);
   \draw[dashed] (1.5,0) -- (1.5,-1);
   \draw[dashed] (4.5,0) -- (4.5,-1);
   \draw[dashed] (6,0) -- (6,-1);
   \draw (0.2,0.51) node {$c_{r-1}$};
   \draw (1.7,0.51) node {$c_{r}$};
   \draw (4.7,0.51) node {$c_{s-1}$};
   \draw (6.2,0.51) node {$c_{s}$};
   \begin{scope}[yshift=-1cm]
    \draw (0.2,0.25) node {$T$};
    \draw (4.7,0.25) node {$T$};
    \draw (-1, 0) -- (7,0);
    \draw (0, 0.1) -- (0,-0.1);
   \draw (1.5, 0.1) -- (1.5,-0.1);
   \draw (4.5, 0.1) -- (4.5,-0.1);
   \draw (6, 0.1) -- (6,-0.1);
   \draw (2.5, 0.1) -- (2.5,-0.1);
   \draw (3.3, 0.1) -- (3.3,-0.1);
   \end{scope}
  \end{tikzpicture}
  \caption{Explanation of $M^{(j)}_{TT} = \widehat{M}^{(j)}_{TT}$.}
  \label{fig:MTT}
 \end{center}
\end{figure}
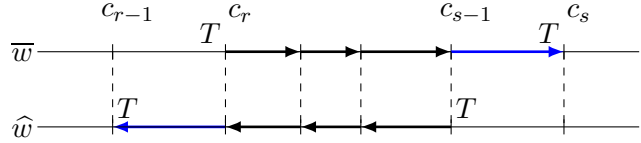

\end{proof}


\begin{lem}
 \label{lem:0orTto0orT}
 For $j=1,2,\ldots,n$, 
 \begin{equation}
  \label{eq:0T}
 M^{(j)}_{0,0}\cup
 M^{(j)}_{0,T}\cup
 \left(M^{(j)}_{T,0}-1\right)\cup
 \left(M^{(j)}_{T,T}-1\right)
 =
 \widehat{M}^{(j)}_{0,0}\cup
 \widehat{M}^{(j)}_{0,T}\cup
 \left(\widehat{M}^{(j)}_{T,0}-1\right)\cup
 \left(\widehat{M}^{(j)}_{T,T}-1\right)
 \end{equation}
\end{lem}
\begin{proof}
By Lemma $\ref{lem:relbarhat}$, we have
\[
 \overline{w}_{c_s}
 +
 \overline{w}_{c_s+1}
 +
 \cdots
 +
 \overline{w}_{c_{s+1}-1}
 =
 \widehat{w}_{c_s}
 +
 \widehat{w}_{c_s+1}
 +
 \cdots
 +
 \widehat{w}_{c_{s+1}-1}
\]
for every $s\in\{0,1,\ldots,p-1\}$.
Since $(\ref{eq:0orT})$ in Lemma $\ref{lem:diff}$ is a bijection,
$i, i+j-1 \in {I}_0\cup {I}_T$ if and only if  $i+1, i+j \in \widehat{I}_0\cup \widehat{I}_T$.
Hence, if $i, i+j-1 \in {I}_0\cup {I}_T$, then there are some $s,t \in \{0,1,\ldots,p-1\}$ such that
$i+1 = c_s$ and $i+j = c_t$. Therefore
\begin{equation}
 \label{eq:0Telement}
 \overline{w}_{i+1} + \overline{w}_{i+2} + \cdots + \overline{w}_{i+j-1}
 =
 \widehat{w}_{i+1} + \widehat{w}_{i+2} + \cdots + \widehat{w}_{i+j-1}
\end{equation}
for all $i\in {I}_0\cup {I}_T$ such that $i+j-1\in {I}_0\cup {I}_T$.
Each element of the multiset of the left (resp. right) hand side of 
$(\ref{eq:0T})$ is expressed as the left (resp. right) hand side of
$(\ref{eq:0Telement})$ and the lemma follows.
\end{proof}

\begin{lem}
 \label{lem:0Tsmallerthanm}
 For $j=1,2,\ldots,m$, 
 \[
 M^{(j)}_{0,T} = \widehat{M}^{(j)}_{0,T}
 \]
\end{lem}
\begin{proof}
 For $j\leq m$,
 $M^{(j)}_{0,T} = \widehat{M}^{(j)}_{0,T} = \emptyset$.
\end{proof}

\begin{lem}
\label{lem:rowcolumnsum}
 Let $j>1$  be an integer and assume that
 $M^{(j-1)}_{a,b} = \widehat{M}^{(j-1)}_{a,b}$
  for   $a,b\in\{0, T, H\}$.
 Then,
 for $a,b\in \{0,T,H\}$,
 \begin{equation}
  \label{eq:rowsum}
  M_{a,0}^{(j)}
 \cup
  \left(M_{a,T}^{(j)}-1\right)
 \cup
  \left(M_{a,H}^{(j)}-1\right)
 =
 \widehat{M}_{a,0}^{(j)}
 \cup
  \left(\widehat{M}_{a,T}^{(j)}-1\right)
 \cup
  \left(\widehat{M}_{a,H}^{(j)}-1\right),
 \end{equation}
 and
 \begin{equation}
  \label{eq:colsum}
  M_{0,b}^{(j)}
 \cup
  \left(M_{T,b}^{(j)}-1\right)
 \cup
  \left(M_{H,b}^{(j)}-1\right)
 =
 \widehat{M}_{0,b}^{(j)}
 \cup
  \left(\widehat{M}_{T,b}^{(j)}-1\right)
 \cup
  \left(\widehat{M}_{H,b}^{(j)}-1\right).
 \end{equation}
\end{lem}
\begin{proof}
 It is clear that
 \[
 M_{0,b}^{(j-1)}
 \cup
 M_{T,b}^{(j-1)}
 \cup
 M_{H,b}^{(j-1)}
 =
 \left\{
 \overline{w}_i
 +
 \overline{w}_{i+1}
 +
 \cdots
 +
 \overline{w}_{i+j-2}
 \,|\,
 i+j-2 \in {I}_b
 \right\},
 \]
 and
 \[
 M_{0,b}^{(j)}
 \cup
 M_{T,b}^{(j)}
 \cup
 M_{H,b}^{(j)}
 =
 \left\{
 \overline{w}_i
 +
 \overline{w}_{i+1}
 +
 \cdots
 +
 \overline{w}_{i+j-1}
 \,|\,
 i+j-1 \in {I}_b
 \right\}.
 \]
 Therefore, we have
 \[
  M_{0,b}^{(j)}
 \cup
  \left(M_{T,b}^{(j)}-1\right)
 \cup
  \left(M_{H,b}^{(j)}-1\right)
 =
 M_{0,b}^{(j-1)}
 \cup
 M_{T,b}^{(j-1)}
 \cup
 M_{H,b}^{(j-1)},
 \]
 and
 \[
 \widehat{M}_{0,b}^{(j)}
 \cup
  \left(\widehat{M}_{T,b}^{(j)}-1\right)
 \cup
  \left(\widehat{M}_{H,b}^{(j)}-1\right)
 =
 \widehat{M}_{0,b}^{(j-1)}
 \cup
 \widehat{M}_{T,b}^{(j-1)}
 \cup
 \widehat{M}_{H,b}^{(j-1)}.
 \]
 From the assumption that
 $M^{(j-1)}_{a,b} = \widehat{M}^{(j-1)}_{a,b}$
 for   $a,b\in\{0, T, H\}$, $(\ref{eq:colsum})$ follows.
 The proof of $(\ref{eq:rowsum})$ is similar.
\end{proof}

\begin{prop}
 \label{prop:smallerthanm}
 For $j=1,2,\ldots,m$ and $a,b\in\{0,T,H\}$,
\begin{equation}
 M_{a,b}^{(j)}=\widehat{M}_{a,b}^{(j)}.
\end{equation}
\end{prop}
\begin{proof}
 We have already shown this for
 $(a,b)=(0,0)$ and $(T,T)$ in Lemma $\ref{lem:00TT}$
 and for $(a,b)=(0,T)$  in Lemma $\ref{lem:0Tsmallerthanm}$.
 We already showed $M_{a,b}^{(1)}=\widehat{M}_{a,b}^{(1)}$ 
 for all $(a,b)\in\{0,H,T\}^2$ in $(\ref{eq:j1})$.
 Now we proceed with induction by assuming $M_{a,b}^{(j-1)}=\widehat{M}_{a,b}^{(j-1)}$
 for all $(a,b)\in\{0,H,T\}^2$.
 By $(\ref{eq:colsum})$, we have
 $M_{H,T}^{(j)} = \widehat{M}_{H,T}^{(j)}$ (shown by the symbol $\spadesuit$ in Figure $\ref{fig:lemmatables}$).
 Then, we have $M_{0,H}^{(j)} = \widehat{M}_{0,H}^{(j)}$ by $(\ref{eq:rowsum})$
 (shown by the symbol $\heartsuit$ in Figure $\ref{fig:lemmatables}$),
 and $M_{T,0}^{(j)} = \widehat{M}_{T,0}^{(j)}$ by Lemma $\ref{lem:0orTto0orT}$
 (shown by the symbol $\clubsuit$ in Figure $\ref{fig:lemmatables}$).
 Then $M_{H,0}^{(j)} = \widehat{M}_{H,0}^{(j)}$ follows from $(\ref{eq:colsum})$
 (shown by $\diamondsuit$).
 Finally $M_{T,H}^{(j)} = \widehat{M}_{T,H}^{(j)}$ and 
 $M_{H,H}^{(j)} = \widehat{M}_{H,H}^{(j)}$ follows from $(\ref{eq:rowsum})$
 (shown by  $\bigstar$).

\begin{figure}[H]
 \begin{center}
  \begin{tikzpicture}
   \foreach \x in {1,2,3}{
   \draw (\x,0) -- (\x,4);
   }
   \foreach \y in {1,2,3}{
   \draw (0,\y) -- (4,\y);
   }
   \draw (0.5,2.5) node {$H$};
   \draw (0.5,1.5) node {$T$};
   \draw (0.5,0.5) node {$0$};
   \draw (1.5,3.5) node {$H$};
   \draw (2.5,3.5) node {$T$};
   \draw (3.5,3.5) node {$0$};
   \draw (3.5,0.5) node {\small Lem.$\ref{lem:00TT}$};
   \draw (2.5,1.5) node {\small Lem.$\ref{lem:00TT}$};
   \draw (2.5,0.5) node {\small Lem.$\ref{lem:0Tsmallerthanm}$};
   \draw (2.5,2.5) node {$\spadesuit$};
   \draw (1.5,0.5) node {$\heartsuit$};
   \draw (3.5,1.5) node {$\clubsuit$};
   \draw (3.5,2.5) node {$\diamondsuit$};
   \draw (1.5,2.5) node {$\bigstar$};
   \draw (1.5,1.5) node {$\bigstar$};

   \begin{scope}[xshift=5cm]
    \foreach \x in {1,2,3}{
    \draw (\x,3) -- (\x,4);
    }
    \draw (1,0) -- (1,3)--(4,3);
    \draw[line width=1.5] (2,0) -- ++(2,0) -- ++ (0,2) -- ++(-2,0) -- cycle;
    \foreach \y in {1,2,3}{
    \draw (0,\y) -- (1,\y);
    }

    \draw (3,1) node {Lem.$\ref{lem:0orTto0orT}$};

    \draw (0.5,2.5) node {$H$};
    \draw (0.5,1.5) node {$T$};
    \draw (0.5,0.5) node {$0$};
    \draw (1.5,3.5) node {$H$};
    \draw (2.5,3.5) node {$T$};
    \draw (3.5,3.5) node {$0$};
   \end{scope}

   \begin{scope}[yshift=-4.5cm]
    \foreach \x in {1,2,3}{
    \draw (\x,3) -- (\x,4);
    }
    \foreach \y in {1,2,3}{
    \draw (0,\y) -- (4,\y);
    \draw[line width=1.5] (1,\y) -- ++(3,0) -- ++(0,-1) -- ++(-3,0) -- cycle ;
    }
    \draw (2.5,0.5) node {Lem.$\ref{lem:rowcolumnsum}$ $(\ref{eq:rowsum})$};
    \draw (2.5,1.5) node {Lem.$\ref{lem:rowcolumnsum}$ $(\ref{eq:rowsum})$};
    \draw (2.5,2.5) node {Lem.$\ref{lem:rowcolumnsum}$ $(\ref{eq:rowsum})$};

    \draw (0.5,2.5) node {$H$};
    \draw (0.5,1.5) node {$T$};
    \draw (0.5,0.5) node {$0$};
    \draw (1.5,3.5) node {$H$};
    \draw (2.5,3.5) node {$T$};
    \draw (3.5,3.5) node {$0$};
   \end{scope}

   \begin{scope}[xshift=5cm,yshift=-4.5cm]
    \foreach \x in {1,2,3}{
    \draw (\x,3) -- (\x,4);
    \draw[line width=1.5] (\x,3) -- ++(0,-3) -- ++(1,0) -- ++(0,3) -- cycle ;
    }
    \foreach \y in {1,2,3}{
    \draw (0,\y) -- (1,\y);
    }
    \draw (1.5,1.5) node {\rotatebox{-90}{Lem.$\ref{lem:rowcolumnsum}$ $(\ref{eq:colsum})$}};
    \draw (2.5,1.5) node {\rotatebox{-90}{Lem.$\ref{lem:rowcolumnsum}$ $(\ref{eq:colsum})$}};
    \draw (3.5,1.5) node {\rotatebox{-90}{Lem.$\ref{lem:rowcolumnsum}$ $(\ref{eq:colsum})$}};

    \draw (0.5,2.5) node {$H$};
    \draw (0.5,1.5) node {$T$};
    \draw (0.5,0.5) node {$0$};
    \draw (1.5,3.5) node {$H$};
    \draw (2.5,3.5) node {$T$};
    \draw (3.5,3.5) node {$0$};

   \end{scope}
  \end{tikzpicture}
  \caption{Illustration of how lemmas are related and used in the proof of Proposition $\ref{prop:smallerthanm}$}
  \label{fig:lemmatables}
 \end{center}
\end{figure}
\end{proof}

To extend Proposition $\ref{prop:smallerthanm}$ for 
$j > m$, we need to extend Lemma $\ref{lem:0Tsmallerthanm}$
for $j>m$. 
The following lemma is the key to this extension. It is proved
by induction on $j$ whose base case has been shown
as Proposition $\ref{prop:smallerthanm}$.

\begin{lem}
 \label{lem:0orTtoT}
 For $j=1,2,\ldots,n$,
 \[
 M^{(j)}_{0,T} \cup \left(M^{(j)}_{T,T}-1\right)= 
 \widehat{M}^{(j)}_{0,T} \cup \left(\widehat{M}^{(j)}_{T,T}-1\right).
 \]
\end{lem}
\begin{proof}
 We have already proved in Proposition $\ref{prop:smallerthanm}$ that
 $M_{a,b}^{(j)}=\widehat{M}_{a,b}^{(j)}$ for $j \leq m$ and $a,b\in\{0,H,T\}$.
 For $j > m$, from the bijections $(\ref{eq:0orT})$ and $(\ref{eq:T})$, we have
 \begin{equation}
  \label{eq:thekey}
   i\in I_0\cup I_T \mbox{ and } i+j-1\in I_T~~
   \Longleftrightarrow~~
   i+j-m \in \widehat{I}_T \mbox{ and } i+1 \in \widehat{I}_0 \cup \widehat{I}_T,
 \end{equation}
 and its {\em reverse version},
 \begin{equation}
  \label{eq:thekeyreverse}
  i\in \widehat{I}_0\cup \widehat{I}_T \mbox{ and } i-j+1\in \widehat{I}_T~~
 \Longleftrightarrow~~
 i-j+m \in I_T \mbox{ and } i-1 \in I_0 \cup I_T.
 \end{equation}
 Therefore, we have
 \begin{eqnarray*}
  M^{(j)}_{0,T} \cup \left(M^{(j)}_{T,T}-1\right)  
   & = &
   \left(\widehat{M}^{(j-m)}_{T,0} \cup \widehat{M}^{(j-m)}_{T,T}\right)+m-1\\
  & = & 
   \left(M^{(j-m)}_{T,0} \cup M^{(j-m)}_{T,T}\right)+m-1\\
  & = & 
  \widehat{M}^{(j)}_{0,T} \cup \left(\widehat{M}^{(j)}_{T,T}-1\right),
 \end{eqnarray*}
 in which the first equality comes from $(\ref{eq:thekey})$ (See Figure $\ref{fig:key}$),
 the second equality comes from the induction hypothesis,
 and the third equality comes from $(\ref{eq:thekeyreverse})$.

\begin{center}
 \begin{figure}[H]
  \begin{center}
   \begin{tikzpicture}
    \draw (-0.2,0) node {$\overline{w}$};
    \draw (-0.2,-1) node {$\widehat{w}$};
    \draw (-0, 0) -- (7,0);
    \draw (1.5, 0.1) -- (1.5,-0.1);
    \draw (6, 0.1) -- (6,-0.1);
    \draw (4.5, 0.1) -- (4.5,-0.1);
    \draw (1.3,0.25) node {$0$};
    \draw (1.3,-0.1) node {or};
    \draw (1.3,-0.5) node {$T$};
    \draw (5.8,0.25) node {$T$};
    
    \draw (2.5, 0.1) -- (2.5,-0.1);
    \draw (3.3, 0.1) -- (3.3,-0.1);

    \draw[line width=1,dotted] (2.5,0) -- (3.3,0);
    \draw[line width=1, >=latex, ->] (3.3,0) -- (4.5,0);
    \draw[line width=1, >=latex, ->,blue] (4.5,0) -- (6,0);
    \draw[line width=1, >=latex, ->] (1.5,0) -- (2.5,0);
    \draw[line width=1, dotted] (3.3,-1) -- (2.5,-1);
    \draw[line width=1, >=latex, <-] (1.5,-1) -- (2.5,-1);
    \draw[line width=1, >=latex, <-] (3.3,-1) -- (4.5,-1);
    \draw[dashed] (2.5, 0.) -- (2.5,-1);
    \draw[dashed] (3.3, 0.) -- (3.3,-1);
    
    \draw[dashed] (1.5,0) -- (1.5,-1);
    \draw[dashed] (4.5,0) -- (4.5,-1);
    \draw[dashed] (6,0) -- (6,-1);
    \draw (1.3,0.6) node {$i$};
    \draw (5.8,0.6) node {\tiny $i+j-1$};
    \begin{scope}[yshift=-1cm]
     \draw (4.7,-0.25) node {$T$};
     \draw (4.7,-0.7) node {\tiny $i+j-m$};
     \draw (-0, 0) -- (7,0);
     \draw (1.5, 0.1) -- (1.5,-0.1);
     \draw (4.5, 0.1) -- (4.5,-0.1);
     \draw (6, 0.1) -- (6,-0.1);
     \draw (2.5, 0.1) -- (2.5,-0.1);
     \draw (3.3, 0.1) -- (3.3,-0.1);
    \draw (1.7,0.35) node {$0$};
    \draw (1.7,0.1) node {or};
    \draw (1.7,-0.3) node {$T$};
    \draw (1.7,-0.6) node {\tiny $i+1$};
    \draw[line width=1, >=latex, <-,blue] (4.5,0) -- (6,0);
    \end{scope}
   \end{tikzpicture}
   \caption{  $M^{(j)}_{0,T} \cup \left(M^{(j)}_{T,T}-1\right)   = 
   \left(\widehat{M}^{(j-m)}_{T,0} \cup \widehat{M}^{(j-m)}_{T,T}\right)+m-1$.
   Note that we need to subtract one from each element of $M^{(j)}_{T,T}$ for the equality to hold.
   }
   \label{fig:key}
  \end{center}
 \end{figure}
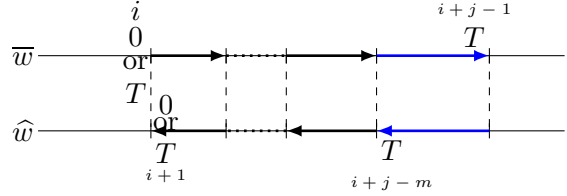
\end{center}
\end{proof}

\begin{proof}[Proof of Theorem $\ref{thm:main}$]
We recall that Lemma $\ref{lem:0orTtoT}$ and Lemma $\ref{lem:00TT}$
hold for $j=1,2,\ldots,n$,
from which we have
\begin{equation}
 \label{eq:0T}
 M_{0,T}^{(j)}= \widehat{M}_{0,T}^{(j)}
\end{equation}
for $j=1,2,\ldots,n$.
The proof is same as the proof of Proposition $\ref{prop:smallerthanm}$ except 
using $(\ref{eq:0T})$ in place of Lemma $\ref{lem:0Tsmallerthanm}$. 
\end{proof}

\subsection{Orbit size and the number of orbits}
In this subsection we show some properties of the
orbit size and the number of orbits of the generalized rotation $\rho$. 
We construct a certain finite set $P$,
and a bijection $\pi:\{0,1\}^n\to P$.
Then we study the orbit structure of the map $\theta = \pi\circ\rho\circ\pi^{-1}:P\to P$ 
which is easier to analyze.

A {\em necklace} of length $n$ is an equivalence class of words over some
alphabet where two words are equivalent if they are (ordinary) rotations of each other.  

Let $w=w_0w_1\cdots w_{n-1}\in\{0,1\}^n$ be a binary word of length $n$
having $k$ zeros.
Suppose that the rising sequence $z_1 < z_2 < \cdots < z_k$ is the indices of zeros
in $w$, that is, $w_i=0$ if and only if $i\in\{z_1,\ldots,z_k\}$.
Then the {\em one run encoding} of $w$ is the sequence
$z_1(z_2-z_1-1)\cdots(z_k-z_{k-1}-1)(n-z_k-1)$,
which we denote ${\rm ore}(w)$.
For example,
\[
 {\rm ore}(1011110)=1(6-1-1)(7-6-1)=140, \mbox{ ~~ }
 {\rm ore}(0111101)=0(5-0-1)(7-5-1)=041.
\]
Let $E_k$ be the ${\rm ore}$'s image of the binary
sequences in $\{0,1\}^n$ having $k$ zeros,
that is,
\[
 E_k = \left\{
 a_0a_1\cdots a_k\,\middle|\,
 a_i\in {\mathbf Z},\,
 a_i\geq 0,\,
 k+\sum_{i=0}^k a_i = n
 \right\},
\]
and define 
\[
 E = \bigcup_{k=0}^nE_k.
\]
Then, since ${\rm ore}\left(\stackrel{a_0}{\overbrace{11\cdots 1}}0\stackrel{a_1}{\overbrace{11\cdots 1}}0\cdots 0\stackrel{a_k}{\overbrace{11\cdots 1}}\right)
=a_0a_1\cdots a_k$,
the map ${\rm ore}:\{0,1\}^n\to E$ is a bijection.

\begin{lem}\label{lem:rotate}
 Suppose that $w\in\{0,1\}^n$ and ${\rm ore}(w)=a_0 a_1 \cdots a_{k}$.
 Then, 
 \[
 {\rm ore}(\rho(w))=\begin{cases}
		     a_1a_2\cdots a_{k} a_0 & \mbox{ if } a_0 < m,\\
		     (a_0-m)a_1 \cdots a_{k-1}(a_{k}+m) & \mbox{ if }a_0 \ge m.
		    \end{cases}
 \]
\end{lem}
\begin{proof}
  If the binary word starts with $j$ consecutive ones followed by a zero and $j
  < m$, then the operation of $\rho$ removes the $j$ ones and the zero from the
  head of the word and then adding the zero and the $j$ ones to the tail of the
  word.

  If the binary word starts with $j \ge m$ consecutive ones then the operation
  of $\rho$ moves $m$ ones from the head to the tail of the word, thereby
  decreasing $a_0$ by $m$ and increasing $a_k$ by $m$.
\end{proof}

For example in table \ref{tab:orbit9} we have the orbit of 1011110
\begin{table}[H]
  \centering
  \begin{tabular}{c|c|c}
    $k$& $\rho^k(w)$ & $\text{ore}(\rho^k(w))$\\
    \hline
    0 & 1011110 & 140 \\
    1 & 1111001 & 401 \\
    2 & 1001111 & 104 \\
    3 & 0111101 & 041 \\
    4 & 1111010 & 410 \\
    5 & 1010111 & 113 \\
    6 & 1011101 & 131 \\
    7 & 1110101 & 311 \\
    8 & 0101111 & 014 \\
  \end{tabular}
  \caption{An orbit of length 9}
  \label{tab:orbit9}
\end{table}



\begin{dfn}
   Given the one run encoding $a_0a_1 \cdots a_{k}$ of a binary word $w$ we
   define the {\em quotient word} ${\rm qw}(w)=q_0 q_1 \cdots
   q_{k}$ and the {\em remainder word} ${\rm rw}(w)=r_0 r_1 \cdots r_{k}$  where
   \[a_i = r_i + m q_i\]
   and
   \[0 \le r_i < m.\]
   We now define the {\em binary quotient word} ${\rm bqw}(w)=b_0 b_1 \dots b_{n_q}$ as
   the binary word whose one run encoding is the quotient word.
 We call the map $\pi: w \mapsto ({\rm rw}(w),{\rm bqw}(w))$ the {\em  encoding map}.
 \end{dfn}

 Note that the binary quotient word, having a one run encoding of length $k+1$,
 has exactly $k$ zeros. 
 Also note that for any remainder word $r_0\cdots r_k$
 and quotient word
 $a_0\cdots a_{k}$ both of length $k+1$, there exists 
 exactly one binary word having $k$ zeros whose one run
 encoding is $a_0\cdots a_{k}$ where $a_i \equiv r_i\pmod{m}$, and so any pair of a
 reminder word of length $k+1$ and binary quotient word having $k$ zeros specify
 a particular binary word (also having $k$ zeros).
 This fact enables us to  construct a 
 finite set $P$ so that the map $\pi:\{0,1\}^n \to P$ is a bijection.
 By considering the effect of $\rho$ mapped on $P$, we obtain
 some properties of the $\rho$-orbits.

 \begin{dfn}
 We define $P_k$ to be $\pi$'s image of the set of binary words in $\{0,1\}^n$ having $k$ zeros,
  that is,
  \begin{eqnarray*}
  P_k & = &\left\{
  (r_0\cdots r_k, b_0\cdots b_{n_q})\in
  \{0,1,\ldots,m-1\}^{k+1}\times\{0,1\}^{n_q}
  \,\middle|\,\right.\\
   & & 
    \hspace{3cm}\left.n_q \geq k,~ k+\sum_{i=0}^{n_q}b_i = n_q+1,~ 
		 k+m\sum_{i=0}^{n_q}b_i+\sum_{i=0}^{k} r_i=n
  \right\},
  \end{eqnarray*}
 and we define
 \[
  P = \bigcup_{k=0}^n P_k.
 \]
  Then we define the map $\theta: P \to P$ by
  \[
   \theta(r_0r_1\cdots r_{k},b_0b_1\cdots b_{n_q}) = 
  \begin{cases}
   (r_1\cdots r_kr_0, b_1\cdots b_{n_q}b_0 ) & \mbox{ if } b_0=0,\\
   (r_0r_1\cdots r_k, b_1\cdots b_{n_q}b_0) & \mbox{ if } b_0=1.
  \end{cases}
  \]
 \end{dfn}

Note that both $\theta:P\to P$ and $\pi:\{0,1\}^n\to P$ are bijections.
Also note that given the pair $\pi(w)=({\rm rw}(w), {\rm bqw}(w))$, $\theta$
rotates both ${\rm rw}(w)$ and ${\rm bqw}(w)$ if ${\rm bqw}(w)_0=0$,
and rotates only ${\rm bqw}(w)$ if ${\rm bqw}(w)_0=1$.

\begin{lem}
\label{lem:commute}
 The following diagram commutes, i.e., $\theta = \pi\rho\pi^{-1}$.
 \begin{figure}[H]
  \begin{center}
   \begin{tikzpicture}
    \node(X1) at (0,2) {$\{0,1\}^n$};
    \node(X2) at (3,2) {$\{0,1\}^n$};
    \node(P1) at (0,0) {$P$};
    \node(P2) at (3,0) {$P$};
    \draw (1.5,2) node [above] {$\rho$};
    \draw (1.5,0) node [below] {$\theta$};
    \draw (0,1) node [left] {$\pi$};
    \draw (3,1) node [right] {$\pi$};
    \draw[->,>=latex] (P1) to (P2);
    \draw[->,>=latex] (X1) to (X2);
    \draw[->,>=latex] (X1) to (P1);
    \draw[->,>=latex] (X2) to (P2);
   \end{tikzpicture}
  \end{center}
 \end{figure}
\end{lem}
\begin{proof}
  This follows directly from lemma \ref{lem:rotate}.  
\end{proof}

Thus we can compute the orbit size of $\rho$ by
computing the orbit size of $\theta$ instead.
\begin{prop}
\label{prop:orbitsize}
Let $w\in\{0,1\}^n$ and $\pi(w)=\left(r_0r_1\cdots r_k,\, b_0b_1\cdots b_{n_q}\right)$.
Let $s$ be the size of the necklace containing  $r_0r_1\cdots r_k$.
That is, we suppose $s$ is the smallest positive integer such that
\[
 r_sr_{s+1}\cdots r_kr_0\cdots r_{s-1} = r_0r_1\cdots r_k.
\]
Also let $t$ be the size of the necklace containing $b_0\cdots b_{n_q}$.
Then the size of $\rho$-orbit of $w$ is $st$.
\end{prop}
\begin{proof}
 Note that
 \[
  r_0\cdots r_k = (r_0\cdots r_{s-1})(r_0\cdots r_{s-1})\cdots (r_0\cdots r_{s-1}),
 \]
 which is a repetition of the subword of length $s$, and
 \[
  b_0\cdots b_{n_q} = (b_0\cdots b_{t-1})(b_0\cdots b_{t-1})\cdots (b_0\cdots b_{t-1}).
 \]
 Therefore, there is a positive integer $g$ such that $k+1=sg$ and hence
 $k$ and $s$ are coprime.
 Let $k'$ be the number of zeros contained in $b_0\cdots b_{t-1}$. Then
 $k'$ is clearly a divisor of $k$ and
\begin{eqnarray*}
 \theta^{t}(\pi(w)) & = &\left(r_{k'}r_{k'+1}\cdots r_{k'-1},\, b_0b_1\cdots b_{n_q}\right)\\
 & = & \left((r_{k''}r_{k''+1}\cdots r_{k''+s-1})(r_{k''}r_{k''+1}\cdots r_{k''+s-1})\cdots(r_{k''}r_{k''+1}\cdots r_{k''+s-1}),\, b_0b_1\cdots b_{n_q}\right)\\
 & = & \left((r_{k''}r_{k''+1}\cdots r_{k''-1})(r_{k''}r_{k''+1}\cdots r_{k''-1})\cdots(r_{k''}r_{k''+1}\cdots r_{k''-1}),\, b_0b_1\cdots b_{n_q}\right)
\end{eqnarray*}
where $k''=k'\pmod{s}$ is considered to be an integer in $\{0,1,\ldots,s-1\}$.
Then, $k''$ and $s$ are coprime, since otherwise $k'$ and $s$ have a common divisor which
is also a common divisor of $k$ and $s$. This contradicts the fact $k$ and $s$ are coprime.
Therefore, we have $\theta^{st}(\pi(w))=\pi(w)$ and
$\theta^f(\pi(w))\neq \pi(w)$ for $f=1,2,\ldots,st-1$.
\end{proof}

Table $\ref{tab:sqw}$ shows an example showing the effect of $\rho$ on $P$.
Using Proposition $\ref{prop:orbitsize}$, we can efficiently
decompose the space $\{0,1\}^n$ into $\rho$-orbits by decomposing $P_k$
into $\theta$-orbits.
Table $\ref{tab:orbdecomp}$ shows an example of $\rho$-orbit decompositions.
Proposition $\ref{prop:orbitsize}$ also gives the maximum size
of a $\rho$-orbits.

\begin{table}[H]  
  \centering
  \begin{tabular}{|c|c|c|c|c|c|}
    \hline
    $e$& $\rho^e(w)$ & $\text{ore}(\rho^e(w))$& $\text{rw}(\rho^e(w))$& $\text{qw}(\rho^e(w))$& $\text{bqw}(\rho^e(w))$\\
    \hline
    0 & 1011110 & 140 & 110 & 010 & 010\\
    1 & 1111001 & 401 & 101 & 100 & 100\\
    2 & 1001111 & 104 & 101 & 001 & 001\\
    3 & 0111101 & 041 & 011 & 010 & 010\\
    4 & 1111010 & 410 & 110 & 100 & 100\\
    5 & 1010111 & 113 & 110 & 001 & 001\\
    6 & 1011101 & 131 & 101 & 010 & 010\\
    7 & 1110101 & 311 & 011 & 100 & 100\\
    8 & 0101111 & 014 & 011 & 001 & 001\\
    \hline
  \end{tabular}
  \caption{An orbit of length 9}
  \label{tab:sqw}
\end{table}

\begin{center}
\begin{table}[H]
 \begin{center}
  \begin{tabular}{|l|l|l|l|l|}
\hline
$k$ & $w$ & ${\rm bqw}(w)$ & ${\rm rw}(w)$ & period length\\
\hline
 0 & 1111111 & 11 & 1 & $1 \times 1=1$\\
1 & 1111101 & 10 & 21 & $2 \times 2=4$\\
1 & 1111110 & 110 & 00 & $3 \times 1=3$\\
2 & 1101101 & 00 & 221 & $1 \times 3=3$\\
2 & 1111010 & 100 & 110 & $3 \times 3=9$\\
2 & 1111100 & 100 & 200 & $3 \times 3=9$\\
3 & 1100110 & 000 & 2020 & $1 \times 2=2$\\
3 & 1101010 & 000 & 2110 & $1 \times 4=4$\\
3 & 1100101 & 000 & 2011 & $1 \times 4=4$\\
3 & 1101100 & 000 & 2200 & $1 \times 4=4$\\
3 & 1101001 & 000 & 2101 & $1 \times 4=4$\\
3 & 1010101 & 000 & 1111 & $1 \times 1=1$\\
3 & 1111000 & 1000 & 1000 & $4 \times 4=16$\\
4 & 1010010 & 0000 & 11010 & $1 \times 5=5$\\
4 & 1101000 & 0000 & 21000 & $1 \times 5=5$\\
4 & 1100100 & 0000 & 20100 & $1 \times 5=5$\\
4 & 1100010 & 0000 & 20010 & $1 \times 5=5$\\
4 & 1100001 & 0000 & 20001 & $1 \times 5=5$\\
4 & 1010100 & 0000 & 11100 & $1 \times 5=5$\\
4 & 1110000 & 10000 & 00000 & $5 \times 1=5$\\
5 & 1100000 & 00000 & 200000 & $1 \times 6=6$\\
5 & 1001000 & 00000 & 101000 & $1 \times 6=6$\\
5 & 1010000 & 00000 & 110000 & $1 \times 6=6$\\
5 & 1000100 & 00000 & 100100 & $1 \times 3=3$\\
6 & 1000000 & 000000 & 1000000 & $1 \times 7=7$\\
7 & 0000000 & 0000000 & 00000000 & $1 \times 1=1$\\
\hline
  \end{tabular}
 \end{center}
 \caption{$\rho$-orbit decomposition of $\{0,1\}^7$ with $m=3$}
 \label{tab:orbdecomp}
\end{table}
\end{center}

\begin{cor}
 The maximum size of a $\rho$-orbit in $\{0,1\}^n$ is
 $\max\left\{(n-m)^2, n\right\}$.
\end{cor}
\begin{proof}
It is clear that the $\rho$-orbit of $w=100\cdots 0$ is of size $n$.
If the binary quotient word ${\rm bqw}(w)$ contains no one, i.e.,
${\rm bqw}(w)=00\cdots 0$,
then the size of the necklace of ${\rm rw}(w)$ is less than or equal to $n$,
and hence the $\rho$-orbit size of $w$ does not exceed $n$.
If ${\rm bqw}(w)$ contains at least one one
and  ${\rm rw}(w)$ contains no one,
then the $\rho$-orbit size of $w$ is $n-m+1 \leq n$.

If both ${\rm bqw}(w)$ and ${\rm rw}(w)$
contain ones
then the lengths of ${\rm bqw}(w)$ and  ${\rm rw}(w)$
are both less than or equal to  $n-m$.
Therefore 
the $\rho$-orbit size of $w$ does not exceed $(n-m)^2$,
which is attained when 
$w=\stackrel{m+1}{\overbrace{11\cdots 1}}\stackrel{n-m-1}{\overbrace{00\cdots 0}}$.

\end{proof}

\section{Toggle dynamical system on $X_N$}
\label{sec:toggle}

Let $N$ be a positive integer not smaller than  $m$.
Let $X_N$ denote the subset of $\{0,1\}^N$ defined by
\[
 X_N=
 \left\{w=w_0w_1\cdots w_{N-1} \in \{0,1\}^N\,|\, w_i+w_{i+1}+\cdots+w_{i+m}\leq 1 \mbox{~for~} i=0,1,\ldots,N-m-1\right\}.
\]
In other words, $X_N$ is the set of the words $w$ of length $N$,
each of whose subwords of the form $w_{[i,i+m]}$ does not contain more than one $1$'s.
We consider the dynamical system $(X_N,\varphi)$ with the state space $X_N$
and the transformation $\varphi:X_N\rightarrow X_N$ defined as follows:
The {\em toggle} map $\tau_i:X_N\rightarrow X_N$ is defined by
\[
 \tau_i(w) = 
 \begin{cases}
  w_0w_1\cdots w_{i-1}(1-w_i)w_{i+1}\cdots w_{N-1}
  &
  w_0w_1\cdots(1-w_i)\cdots w_{N-1} \in X_N,\\
  w
  &
  w_0w_1\cdots(1-w_i)\cdots w_{N-1} \not\in X_N,
 \end{cases}
\]
and $\varphi = \tau_{N-1}\circ\tau_{N-2}\circ\cdots\circ\tau_0$.
It is clear that every $\tau_i$ is a bijections from $X_N$
to itself, and so is $\varphi$.
Therefore $\varphi$ decomposes $X_N$ into $\varphi$-orbits.

Joseph and Roby\cite{JR} studied the dynamical system
$(X_N,\varphi)$ for $m=1$
and showed some surprising properties.
In particular, they showed the symmetry of the digit sum of orbits:
When $m=1$, for every $w\in X_N$ and $j\in\{0,1,\ldots,N-1\}$
\begin{equation}
 \label{eq:main}
 \sum_{k=0}^{p-1}\varphi^k(w)_j
 =
 \sum_{k=0}^{p-1}\varphi^k(w)_{N-1-j},
\end{equation}
where $p$ is the length of the $\varphi$-orbit of $w$
and $\varphi^k(w)_j$ is the $j$-th digit of the word $\varphi^k(w)$.

The key idea of the proof of
$(\ref{eq:main})$ for $m=1$
by Joseph and Roby\cite{JR} 
is the reduction of the toggle dynamical
system to the rotation of the bit sequences
by using the notion of the {\em snakes}.
We show that $(\ref{eq:main})$ holds for general $m$
by reducing it to a dynamical
system driven by the generalized rotations
which has been discussed in previous sections.

Let $\{0,1\}^*$ denote the set of finite words over the alphabet $\{0,1\}$,
and let $w=w_0w_1w_2\ldots w_{|w|-1}\in\{0,1\}^*$ be a finite word.
Then, we define $a(w)=\sum_{i=0}^{|w|-1}w_i$ and $b(w)=|w|-a(w)$, that is,
$a(w)$ is the number of $1$'s in $w$ and $b(w)$ the number of $0$'s.
Define the subset $Y_{n}$ of $\{0,1\}^*$ by
\[
 Y_{n} =\{w\in\{0,1\}^*\,|\, a(w)+(m+1)b(w)=n\}.
\]
\begin{exm}
When $m=3$, we have
$Y_3=\{111\}$, $Y_4=\{1111, 0\}$, 
$Y_5=\{11111, 10, 01\}$, and $Y_6=\{111111,110, 101, 011\}$.
\end{exm}

\begin{table}[H]
 \begin{center}
  \begin{tikzpicture}[yscale=0.4, xscale=0.6]
\draw (-3,31.5) -- (14,31.5);
\draw (-3,.5) -- (14,.5);
\draw (-1,32) node {$j$};
\node[color=red] (0-0) at (0,31) [circle,draw,inner sep=0] {1};
\node[color=red] (1-1) at (1,30) [circle,draw,inner sep=0] {1};
\draw[color=red] (0-0) edge (1-1);
\node[color=red] (1-5) at (5,30) [circle,draw,inner sep=0] {1};
\draw[color=red] (1-1) edge (1-5);
\node[color=red] (2-6) at (6,29) [circle,draw,inner sep=0] {1};
\draw[color=red] (1-5) edge (2-6);
\node[color=red] (3-7) at (7,28) [circle,draw,inner sep=0] {1};
\draw[color=red] (2-6) edge (3-7);
\node[color=red] (4-8) at (8,27) [circle,draw,inner sep=0] {1};
\draw[color=red] (3-7) edge (4-8);
\node[color=red] (5-9) at (9,26) [circle,draw,inner sep=0] {1};
\draw[color=red] (4-8) edge (5-9);
\node[color=red] (5-13) at (13,26) [circle,draw,inner sep=0] {1};
\draw[color=red] (5-9) edge (5-13);
\node[color=red] (3-0) at (0,28) [circle,draw,inner sep=0] {1};
\node[color=red] (4-1) at (1,27) [circle,draw,inner sep=0] {1};
\draw[color=red] (3-0) edge (4-1);
\node[color=red] (5-2) at (2,26) [circle,draw,inner sep=0] {1};
\draw[color=red] (4-1) edge (5-2);
\node[color=red] (6-3) at (3,25) [circle,draw,inner sep=0] {1};
\draw[color=red] (5-2) edge (6-3);
\node[color=red] (7-4) at (4,24) [circle,draw,inner sep=0] {1};
\draw[color=red] (6-3) edge (7-4);
\node[color=red] (7-8) at (8,24) [circle,draw,inner sep=0] {1};
\draw[color=red] (7-4) edge (7-8);
\node[color=red] (7-12) at (12,24) [circle,draw,inner sep=0] {1};
\draw[color=red] (7-8) edge (7-12);
\node[color=red] (8-13) at (13,23) [circle,draw,inner sep=0] {1};
\draw[color=red] (7-12) edge (8-13);
\node[color=red] (8-0) at (0,23) [circle,draw,inner sep=0] {1};
\node[color=red] (9-1) at (1,22) [circle,draw,inner sep=0] {1};
\draw[color=red] (8-0) edge (9-1);
\node[color=red] (9-5) at (5,22) [circle,draw,inner sep=0] {1};
\draw[color=red] (9-1) edge (9-5);
\node[color=red] (9-9) at (9,22) [circle,draw,inner sep=0] {1};
\draw[color=red] (9-5) edge (9-9);
\node[color=red] (10-10) at (10,21) [circle,draw,inner sep=0] {1};
\draw[color=red] (9-9) edge (10-10);
\node[color=red] (11-11) at (11,20) [circle,draw,inner sep=0] {1};
\draw[color=red] (10-10) edge (11-11);
\node[color=red] (12-12) at (12,19) [circle,draw,inner sep=0] {1};
\draw[color=red] (11-11) edge (12-12);
\node[color=red] (13-13) at (13,18) [circle,draw,inner sep=0] {1};
\draw[color=red] (12-12) edge (13-13);
\node[color=red] (11-0) at (0,20) [circle,draw,inner sep=0] {1};
\node[color=red] (11-4) at (4,20) [circle,draw,inner sep=0] {1};
\draw[color=red] (11-0) edge (11-4);
\node[color=red] (12-5) at (5,19) [circle,draw,inner sep=0] {1};
\draw[color=red] (11-4) edge (12-5);
\node[color=red] (13-6) at (6,18) [circle,draw,inner sep=0] {1};
\draw[color=red] (12-5) edge (13-6);
\node[color=red] (14-7) at (7,17) [circle,draw,inner sep=0] {1};
\draw[color=red] (13-6) edge (14-7);
\node[color=red] (15-8) at (8,16) [circle,draw,inner sep=0] {1};
\draw[color=red] (14-7) edge (15-8);
\node[color=red] (15-12) at (12,16) [circle,draw,inner sep=0] {1};
\draw[color=red] (15-8) edge (15-12);
\node[color=red] (16-13) at (13,15) [circle,draw,inner sep=0] {1};
\draw[color=red] (15-12) edge (16-13);
\node[color=red] (13-0) at (0,18) [circle,draw,inner sep=0] {1};
\node[color=red] (14-1) at (1,17) [circle,draw,inner sep=0] {1};
\draw[color=red] (13-0) edge (14-1);
\node[color=red] (15-2) at (2,16) [circle,draw,inner sep=0] {1};
\draw[color=red] (14-1) edge (15-2);
\node[color=red] (16-3) at (3,15) [circle,draw,inner sep=0] {1};
\draw[color=red] (15-2) edge (16-3);
\node[color=red] (17-4) at (4,14) [circle,draw,inner sep=0] {1};
\draw[color=red] (16-3) edge (17-4);
\node[color=red] (17-8) at (8,14) [circle,draw,inner sep=0] {1};
\draw[color=red] (17-4) edge (17-8);
\node[color=red] (18-9) at (9,13) [circle,draw,inner sep=0] {1};
\draw[color=red] (17-8) edge (18-9);
\node[color=red] (18-13) at (13,13) [circle,draw,inner sep=0] {1};
\draw[color=red] (18-9) edge (18-13);
\node[color=red] (18-0) at (0,13) [circle,draw,inner sep=0] {1};
\node[color=red] (19-1) at (1,12) [circle,draw,inner sep=0] {1};
\draw[color=red] (18-0) edge (19-1);
\node[color=red] (19-5) at (5,12) [circle,draw,inner sep=0] {1};
\draw[color=red] (19-1) edge (19-5);
\node[color=red] (20-6) at (6,11) [circle,draw,inner sep=0] {1};
\draw[color=red] (19-5) edge (20-6);
\node[color=red] (20-10) at (10,11) [circle,draw,inner sep=0] {1};
\draw[color=red] (20-6) edge (20-10);
\node[color=red] (21-11) at (11,10) [circle,draw,inner sep=0] {1};
\draw[color=red] (20-10) edge (21-11);
\node[color=red] (22-12) at (12,9) [circle,draw,inner sep=0] {1};
\draw[color=red] (21-11) edge (22-12);
\node[color=red] (23-13) at (13,8) [circle,draw,inner sep=0] {1};
\draw[color=red] (22-12) edge (23-13);
\node[color=red] (21-0) at (0,10) [circle,draw,inner sep=0] {1};
\node[color=red] (22-1) at (1,9) [circle,draw,inner sep=0] {1};
\draw[color=red] (21-0) edge (22-1);
\node[color=red] (22-5) at (5,9) [circle,draw,inner sep=0] {1};
\draw[color=red] (22-1) edge (22-5);
\node[color=red] (23-6) at (6,8) [circle,draw,inner sep=0] {1};
\draw[color=red] (22-5) edge (23-6);
\node[color=red] (24-7) at (7,7) [circle,draw,inner sep=0] {1};
\draw[color=red] (23-6) edge (24-7);
\node[color=red] (25-8) at (8,6) [circle,draw,inner sep=0] {1};
\draw[color=red] (24-7) edge (25-8);
\node[color=red] (25-12) at (12,6) [circle,draw,inner sep=0] {1};
\draw[color=red] (25-8) edge (25-12);
\node[color=red] (26-13) at (13,5) [circle,draw,inner sep=0] {1};
\draw[color=red] (25-12) edge (26-13);
\node[color=red] (24-0) at (0,7) [circle,draw,inner sep=0] {1};
\node[color=red] (25-1) at (1,6) [circle,draw,inner sep=0] {1};
\draw[color=red] (24-0) edge (25-1);
\node[color=red] (26-2) at (2,5) [circle,draw,inner sep=0] {1};
\draw[color=red] (25-1) edge (26-2);
\node[color=red] (27-3) at (3,4) [circle,draw,inner sep=0] {1};
\draw[color=red] (26-2) edge (27-3);
\node[color=red] (27-7) at (7,4) [circle,draw,inner sep=0] {1};
\draw[color=red] (27-3) edge (27-7);
\node[color=red] (28-8) at (8,3) [circle,draw,inner sep=0] {1};
\draw[color=red] (27-7) edge (28-8);
\node[color=red] (28-12) at (12,3) [circle,draw,inner sep=0] {1};
\draw[color=red] (28-8) edge (28-12);
\node[color=red] (29-13) at (13,2) [circle,draw,inner sep=0] {1};
\draw[color=red] (28-12) edge (29-13);
\node[color=red] (29-0) at (0,2) [circle,draw,inner sep=0] {1};
\node[color=red] (29-4) at (4,2) [circle,draw,inner sep=0] {1};
\draw[color=red] (29-0) edge (29-4);
\node[color=red] (30-5) at (5,1) [circle,draw,inner sep=0] {1};
\draw[color=red] (29-4) edge (30-5);
\node[color=red] (30-9) at (9,1) [circle,draw,inner sep=0] {1};
\draw[color=red] (30-5) edge (30-9);
\node[color=red] (0-10) at (10,31) [circle,draw,inner sep=0] {1};
\node[color=red] (1-11) at (11,30) [circle,draw,inner sep=0] {1};
\draw[color=red] (0-10) edge (1-11);
\node[color=red] (2-12) at (12,29) [circle,draw,inner sep=0] {1};
\draw[color=red] (1-11) edge (2-12);
\node[color=red] (3-13) at (13,28) [circle,draw,inner sep=0] {1};
\draw[color=red] (2-12) edge (3-13);
\node[] at (-2,31)  {$\varphi^{0}(w)$};
\node[] at (1,31)  {0};
\node[] at (2,31)  {0};
\node[] at (3,31)  {0};
\node[] at (4,31)  {0};
\node[] at (5,31)  {0};
\node[] at (6,31)  {0};
\node[] at (7,31)  {0};
\node[] at (8,31)  {0};
\node[] at (9,31)  {0};
\node[] at (11,31)  {0};
\node[] at (12,31)  {0};
\node[] at (13,31)  {0};
\node[] at (-2,30)  {$\varphi^{1}(w)$};
\node[] at (0,30)  {0};
\node[] at (2,30)  {0};
\node[] at (3,30)  {0};
\node[] at (4,30)  {0};
\node[] at (6,30)  {0};
\node[] at (7,30)  {0};
\node[] at (8,30)  {0};
\node[] at (9,30)  {0};
\node[] at (10,30)  {0};
\node[] at (12,30)  {0};
\node[] at (13,30)  {0};
\node[] at (-2,29)  {$\varphi^{2}(w)$};
\node[] at (0,29)  {0};
\node[] at (1,29)  {0};
\node[] at (2,29)  {0};
\node[] at (3,29)  {0};
\node[] at (4,29)  {0};
\node[] at (5,29)  {0};
\node[] at (7,29)  {0};
\node[] at (8,29)  {0};
\node[] at (9,29)  {0};
\node[] at (10,29)  {0};
\node[] at (11,29)  {0};
\node[] at (13,29)  {0};
\node[] at (-2,28)  {$\varphi^{3}(w)$};
\node[] at (1,28)  {0};
\node[] at (2,28)  {0};
\node[] at (3,28)  {0};
\node[] at (4,28)  {0};
\node[] at (5,28)  {0};
\node[] at (6,28)  {0};
\node[] at (8,28)  {0};
\node[] at (9,28)  {0};
\node[] at (10,28)  {0};
\node[] at (11,28)  {0};
\node[] at (12,28)  {0};
\node[] at (-2,27)  {$\varphi^{4}(w)$};
\node[] at (0,27)  {0};
\node[] at (2,27)  {0};
\node[] at (3,27)  {0};
\node[] at (4,27)  {0};
\node[] at (5,27)  {0};
\node[] at (6,27)  {0};
\node[] at (7,27)  {0};
\node[] at (9,27)  {0};
\node[] at (10,27)  {0};
\node[] at (11,27)  {0};
\node[] at (12,27)  {0};
\node[] at (13,27)  {0};
\node[] at (-2,26)  {$\varphi^{5}(w)$};
\node[] at (0,26)  {0};
\node[] at (1,26)  {0};
\node[] at (3,26)  {0};
\node[] at (4,26)  {0};
\node[] at (5,26)  {0};
\node[] at (6,26)  {0};
\node[] at (7,26)  {0};
\node[] at (8,26)  {0};
\node[] at (10,26)  {0};
\node[] at (11,26)  {0};
\node[] at (12,26)  {0};
\node[] at (-2,25)  {$\varphi^{6}(w)$};
\node[] at (0,25)  {0};
\node[] at (1,25)  {0};
\node[] at (2,25)  {0};
\node[] at (4,25)  {0};
\node[] at (5,25)  {0};
\node[] at (6,25)  {0};
\node[] at (7,25)  {0};
\node[] at (8,25)  {0};
\node[] at (9,25)  {0};
\node[] at (10,25)  {0};
\node[] at (11,25)  {0};
\node[] at (12,25)  {0};
\node[] at (13,25)  {0};
\node[] at (-2,24)  {$\varphi^{7}(w)$};
\node[] at (0,24)  {0};
\node[] at (1,24)  {0};
\node[] at (2,24)  {0};
\node[] at (3,24)  {0};
\node[] at (5,24)  {0};
\node[] at (6,24)  {0};
\node[] at (7,24)  {0};
\node[] at (9,24)  {0};
\node[] at (10,24)  {0};
\node[] at (11,24)  {0};
\node[] at (13,24)  {0};
\node[] at (-2,23)  {$\varphi^{8}(w)$};
\node[] at (1,23)  {0};
\node[] at (2,23)  {0};
\node[] at (3,23)  {0};
\node[] at (4,23)  {0};
\node[] at (5,23)  {0};
\node[] at (6,23)  {0};
\node[] at (7,23)  {0};
\node[] at (8,23)  {0};
\node[] at (9,23)  {0};
\node[] at (10,23)  {0};
\node[] at (11,23)  {0};
\node[] at (12,23)  {0};
\node[] at (-2,22)  {$\varphi^{9}(w)$};
\node[] at (0,22)  {0};
\node[] at (2,22)  {0};
\node[] at (3,22)  {0};
\node[] at (4,22)  {0};
\node[] at (6,22)  {0};
\node[] at (7,22)  {0};
\node[] at (8,22)  {0};
\node[] at (10,22)  {0};
\node[] at (11,22)  {0};
\node[] at (12,22)  {0};
\node[] at (13,22)  {0};
\node[] at (-2,21)  {$\varphi^{10}(w)$};
\node[] at (0,21)  {0};
\node[] at (1,21)  {0};
\node[] at (2,21)  {0};
\node[] at (3,21)  {0};
\node[] at (4,21)  {0};
\node[] at (5,21)  {0};
\node[] at (6,21)  {0};
\node[] at (7,21)  {0};
\node[] at (8,21)  {0};
\node[] at (9,21)  {0};
\node[] at (11,21)  {0};
\node[] at (12,21)  {0};
\node[] at (13,21)  {0};
\node[] at (-2,20)  {$\varphi^{11}(w)$};
\node[] at (1,20)  {0};
\node[] at (2,20)  {0};
\node[] at (3,20)  {0};
\node[] at (5,20)  {0};
\node[] at (6,20)  {0};
\node[] at (7,20)  {0};
\node[] at (8,20)  {0};
\node[] at (9,20)  {0};
\node[] at (10,20)  {0};
\node[] at (12,20)  {0};
\node[] at (13,20)  {0};
\node[] at (-2,19)  {$\varphi^{12}(w)$};
\node[] at (0,19)  {0};
\node[] at (1,19)  {0};
\node[] at (2,19)  {0};
\node[] at (3,19)  {0};
\node[] at (4,19)  {0};
\node[] at (6,19)  {0};
\node[] at (7,19)  {0};
\node[] at (8,19)  {0};
\node[] at (9,19)  {0};
\node[] at (10,19)  {0};
\node[] at (11,19)  {0};
\node[] at (13,19)  {0};
\node[] at (-2,18)  {$\varphi^{13}(w)$};
\node[] at (1,18)  {0};
\node[] at (2,18)  {0};
\node[] at (3,18)  {0};
\node[] at (4,18)  {0};
\node[] at (5,18)  {0};
\node[] at (7,18)  {0};
\node[] at (8,18)  {0};
\node[] at (9,18)  {0};
\node[] at (10,18)  {0};
\node[] at (11,18)  {0};
\node[] at (12,18)  {0};
\node[] at (-2,17)  {$\varphi^{14}(w)$};
\node[] at (0,17)  {0};
\node[] at (2,17)  {0};
\node[] at (3,17)  {0};
\node[] at (4,17)  {0};
\node[] at (5,17)  {0};
\node[] at (6,17)  {0};
\node[] at (8,17)  {0};
\node[] at (9,17)  {0};
\node[] at (10,17)  {0};
\node[] at (11,17)  {0};
\node[] at (12,17)  {0};
\node[] at (13,17)  {0};
\node[] at (-2,16)  {$\varphi^{15}(w)$};
\node[] at (0,16)  {0};
\node[] at (1,16)  {0};
\node[] at (3,16)  {0};
\node[] at (4,16)  {0};
\node[] at (5,16)  {0};
\node[] at (6,16)  {0};
\node[] at (7,16)  {0};
\node[] at (9,16)  {0};
\node[] at (10,16)  {0};
\node[] at (11,16)  {0};
\node[] at (13,16)  {0};
\node[] at (-2,15)  {$\varphi^{16}(w)$};
\node[] at (0,15)  {0};
\node[] at (1,15)  {0};
\node[] at (2,15)  {0};
\node[] at (4,15)  {0};
\node[] at (5,15)  {0};
\node[] at (6,15)  {0};
\node[] at (7,15)  {0};
\node[] at (8,15)  {0};
\node[] at (9,15)  {0};
\node[] at (10,15)  {0};
\node[] at (11,15)  {0};
\node[] at (12,15)  {0};
\node[] at (-2,14)  {$\varphi^{17}(w)$};
\node[] at (0,14)  {0};
\node[] at (1,14)  {0};
\node[] at (2,14)  {0};
\node[] at (3,14)  {0};
\node[] at (5,14)  {0};
\node[] at (6,14)  {0};
\node[] at (7,14)  {0};
\node[] at (9,14)  {0};
\node[] at (10,14)  {0};
\node[] at (11,14)  {0};
\node[] at (12,14)  {0};
\node[] at (13,14)  {0};
\node[] at (-2,13)  {$\varphi^{18}(w)$};
\node[] at (1,13)  {0};
\node[] at (2,13)  {0};
\node[] at (3,13)  {0};
\node[] at (4,13)  {0};
\node[] at (5,13)  {0};
\node[] at (6,13)  {0};
\node[] at (7,13)  {0};
\node[] at (8,13)  {0};
\node[] at (10,13)  {0};
\node[] at (11,13)  {0};
\node[] at (12,13)  {0};
\node[] at (-2,12)  {$\varphi^{19}(w)$};
\node[] at (0,12)  {0};
\node[] at (2,12)  {0};
\node[] at (3,12)  {0};
\node[] at (4,12)  {0};
\node[] at (6,12)  {0};
\node[] at (7,12)  {0};
\node[] at (8,12)  {0};
\node[] at (9,12)  {0};
\node[] at (10,12)  {0};
\node[] at (11,12)  {0};
\node[] at (12,12)  {0};
\node[] at (13,12)  {0};
\node[] at (-2,11)  {$\varphi^{20}(w)$};
\node[] at (0,11)  {0};
\node[] at (1,11)  {0};
\node[] at (2,11)  {0};
\node[] at (3,11)  {0};
\node[] at (4,11)  {0};
\node[] at (5,11)  {0};
\node[] at (7,11)  {0};
\node[] at (8,11)  {0};
\node[] at (9,11)  {0};
\node[] at (11,11)  {0};
\node[] at (12,11)  {0};
\node[] at (13,11)  {0};
\node[] at (-2,10)  {$\varphi^{21}(w)$};
\node[] at (1,10)  {0};
\node[] at (2,10)  {0};
\node[] at (3,10)  {0};
\node[] at (4,10)  {0};
\node[] at (5,10)  {0};
\node[] at (6,10)  {0};
\node[] at (7,10)  {0};
\node[] at (8,10)  {0};
\node[] at (9,10)  {0};
\node[] at (10,10)  {0};
\node[] at (12,10)  {0};
\node[] at (13,10)  {0};
\node[] at (-2,9)  {$\varphi^{22}(w)$};
\node[] at (0,9)  {0};
\node[] at (2,9)  {0};
\node[] at (3,9)  {0};
\node[] at (4,9)  {0};
\node[] at (6,9)  {0};
\node[] at (7,9)  {0};
\node[] at (8,9)  {0};
\node[] at (9,9)  {0};
\node[] at (10,9)  {0};
\node[] at (11,9)  {0};
\node[] at (13,9)  {0};
\node[] at (-2,8)  {$\varphi^{23}(w)$};
\node[] at (0,8)  {0};
\node[] at (1,8)  {0};
\node[] at (2,8)  {0};
\node[] at (3,8)  {0};
\node[] at (4,8)  {0};
\node[] at (5,8)  {0};
\node[] at (7,8)  {0};
\node[] at (8,8)  {0};
\node[] at (9,8)  {0};
\node[] at (10,8)  {0};
\node[] at (11,8)  {0};
\node[] at (12,8)  {0};
\node[] at (-2,7)  {$\varphi^{24}(w)$};
\node[] at (1,7)  {0};
\node[] at (2,7)  {0};
\node[] at (3,7)  {0};
\node[] at (4,7)  {0};
\node[] at (5,7)  {0};
\node[] at (6,7)  {0};
\node[] at (8,7)  {0};
\node[] at (9,7)  {0};
\node[] at (10,7)  {0};
\node[] at (11,7)  {0};
\node[] at (12,7)  {0};
\node[] at (13,7)  {0};
\node[] at (-2,6)  {$\varphi^{25}(w)$};
\node[] at (0,6)  {0};
\node[] at (2,6)  {0};
\node[] at (3,6)  {0};
\node[] at (4,6)  {0};
\node[] at (5,6)  {0};
\node[] at (6,6)  {0};
\node[] at (7,6)  {0};
\node[] at (9,6)  {0};
\node[] at (10,6)  {0};
\node[] at (11,6)  {0};
\node[] at (13,6)  {0};
\node[] at (-2,5)  {$\varphi^{26}(w)$};
\node[] at (0,5)  {0};
\node[] at (1,5)  {0};
\node[] at (3,5)  {0};
\node[] at (4,5)  {0};
\node[] at (5,5)  {0};
\node[] at (6,5)  {0};
\node[] at (7,5)  {0};
\node[] at (8,5)  {0};
\node[] at (9,5)  {0};
\node[] at (10,5)  {0};
\node[] at (11,5)  {0};
\node[] at (12,5)  {0};
\node[] at (-2,4)  {$\varphi^{27}(w)$};
\node[] at (0,4)  {0};
\node[] at (1,4)  {0};
\node[] at (2,4)  {0};
\node[] at (4,4)  {0};
\node[] at (5,4)  {0};
\node[] at (6,4)  {0};
\node[] at (8,4)  {0};
\node[] at (9,4)  {0};
\node[] at (10,4)  {0};
\node[] at (11,4)  {0};
\node[] at (12,4)  {0};
\node[] at (13,4)  {0};
\node[] at (-2,3)  {$\varphi^{28}(w)$};
\node[] at (0,3)  {0};
\node[] at (1,3)  {0};
\node[] at (2,3)  {0};
\node[] at (3,3)  {0};
\node[] at (4,3)  {0};
\node[] at (5,3)  {0};
\node[] at (6,3)  {0};
\node[] at (7,3)  {0};
\node[] at (9,3)  {0};
\node[] at (10,3)  {0};
\node[] at (11,3)  {0};
\node[] at (13,3)  {0};
\node[] at (-2,2)  {$\varphi^{29}(w)$};
\node[] at (1,2)  {0};
\node[] at (2,2)  {0};
\node[] at (3,2)  {0};
\node[] at (5,2)  {0};
\node[] at (6,2)  {0};
\node[] at (7,2)  {0};
\node[] at (8,2)  {0};
\node[] at (9,2)  {0};
\node[] at (10,2)  {0};
\node[] at (11,2)  {0};
\node[] at (12,2)  {0};
\node[] at (-2,1)  {$\varphi^{30}(w)$};
\node[] at (0,1)  {0};
\node[] at (1,1)  {0};
\node[] at (2,1)  {0};
\node[] at (3,1)  {0};
\node[] at (4,1)  {0};
\node[] at (6,1)  {0};
\node[] at (7,1)  {0};
\node[] at (8,1)  {0};
\node[] at (10,1)  {0};
\node[] at (11,1)  {0};
\node[] at (12,1)  {0};
\node[] at (13,1)  {0};
\node at (0,32) {$0$};
\node at (0,0) {$9$};
\node at (1,32) {$1$};
\node at (1,0) {$7$};
\node at (2,32) {$2$};
\node at (2,0) {$3$};
\node at (3,32) {$3$};
\node at (3,0) {$3$};
\node at (4,32) {$4$};
\node at (4,0) {$4$};
\node at (5,32) {$5$};
\node at (5,0) {$6$};
\node at (6,32) {$6$};
\node at (6,0) {$4$};
\node at (7,32) {$7$};
\node at (7,0) {$4$};
\node at (8,32) {$8$};
\node at (8,0) {$6$};
\node at (9,32) {$9$};
\node at (9,0) {$4$};
\node at (10,32) {$10$};
\node at (10,0) {$3$};
\node at (11,32) {$11$};
\node at (11,0) {$3$};
\node at (12,32) {$12$};
\node at (12,0) {$7$};
\node at (13,32) {$13$};
\node at (13,0) {$9$};
\end{tikzpicture}

 \end{center}
\caption{Orbit board of $10000000001000\in X_{14}$.
 A snake consists of positions of $1$'s which are marked by the circles.}
\label{tab:snake}
\end{table}

By modifying the argument by Joseph and Roby\cite{JR}
using a notion Haddadan\cite{haddadan2014some} dubs {\em snakes},
we construct a bijection between the orbits of $(X_{N+1,m},\varphi)$ and
those of $(Y_N, \rho)$.

\begin{exm}
 When $m=3$,
 $(Y_6,\rho)$ has three orbits $:$
 \[
  Y_6 = \{111111\} \cup \{110,011\} \cup \{101\},
 \]
 and so does $(X_7,\varphi):$
 \begin{eqnarray*}
  X_7 & = &\{1000010,0100001,0010000,0001000,0000100\}  \\
  & &  \cup
 \{1000100,0000010,1000001,0100000,0010001,0000000\}
 \cup
 \{1000000,0100010,0000001\}.
 \end{eqnarray*}
 We construct an explicit bijection between these
 sets of orbits.
\end{exm}

\begin{dfn}
Let $S\in X_{N+1}$ be a word and $q$ be the length of $\varphi$-orbit of $S$.
Then, define the orbit board $(S(i,j))_{0\leq i < q, 0\leq j\leq N}$ for the word $S$,
by
\[
 S(i,j) = \varphi^i(w)_j,
\]
where we consider $i$ to be $\bmod q$, but $j$ is not considered to be $\bmod$ anything.
\end{dfn}

\begin{lem}
 \label{lem:snake}
 \begin{enumerate}
  \item When $S(i,j)=1$ and $j\neq N-1$,
	either $S(i,j+m+1)=1$ or $S(i+1,j+1)=1$, and never both.
  \item When $S(i,j)=1$ and $j\neq 0$,
	either $S(i,j-m-1)=1$ or $S(i-1,j-1)=1$, and never both.
 \end{enumerate}
\end{lem}
\begin{proof}
 \begin{enumerate}
  \item 
	Suppose that  $S(i,j)=1$, that is, $\varphi^i(w)_j=1$. Then,
	$S(i+1,k)$ is determined after sequentially applying the toggle map $\tau_0,\tau_1,\ldots,\tau_k$
	to $\varphi^i(w)$.
	Thus, it is obvious that $S(i+1,k)=0$ for $j-m \leq k \leq j$.
	If $S(i,j+m+1)=\varphi^i(w)_{j+m+1}=1$, then  
	\[
	S(i+1,j+1) = \tau_{j+1}\left(\tau_{j}\circ\cdots\circ\tau_0\circ\varphi^i(w)\right)_{j+1}=0.
	\]
	See the left part of Figure $\ref{fig:10m1}$.
	If $S(i,j+m+1)=0$, then  
	$S(i+1,j+1) = 1$.
	See the right part of Figure $\ref{fig:10m1}$.

 \begin{center}
  \begin{figure}[H]
   \begin{center}
    \begin{tikzpicture}[scale=0.4]
     \draw(0,0) -- (11,0);
     \draw(0,0) -- (0,-7);
     \foreach \x in {1,2,...,10}{
     \draw[gray,line width=0.3](\x,0) -- (\x,-7);
     }
     \foreach \y in {1,2,...,6}{
     \draw[gray,line width=0.3](0,-\y) -- (11,-\y);
     }
     \draw (6.5,0) node[above] {\tiny $j$};
     \draw (0,-3.5) node[left] {\tiny $i$};
     \draw (10.5,0) node[above] {\tiny $j+m+1$};
     \draw[dashed] (10.5,0) -- (10.5,-3);
     \draw[dashed] (6.5,0) -- (6.5,-3);
     \draw[dashed] (0,-3.5) -- (6,-3.5);
     \draw (6.5,-3.5) node {$1$};
     \draw (10.5,-3.5) node {$1$};
     \draw (3.5,-4.5) node {$0$};
     \draw (4.5,-4.5) node {$0$};
     \draw (5.5,-4.5) node {$0$};
     \draw (6.5,-4.5) node {$0$};
     \draw (7.5,-4.5) node {$0$};
     \draw (8.5,-4.5) node {$0$};
     \draw (9.5,-4.5) node {$0$};
     \draw (10.5,-4.5) node {$0$};
     \draw[<->](7,-3.5)--(10,-3.5);
     \draw (8.5,-3.5)node[above]{\small$m$};
     \draw[<->](3,-3.5)--(6,-3.5);
     \draw (4.5,-2.5)node[below]{\small$m$};
     \draw[red,line width=1] (6,-3) -- ++(1,0) -- ++(0,-1) -- ++(-1,0) -- ++(0,1);
     \draw[red,line width=1] (10,-3) -- ++(1,0) -- ++(0,-1) -- ++(-1,0) -- ++(0,1);
     \begin{scope}[xshift=15cm]
           \draw(0,0) -- (11,0);
     \draw(0,0) -- (0,-7);
     \foreach \x in {1,2,...,10}{
     \draw[gray,line width=0.3](\x,0) -- (\x,-7);
     }
     \foreach \y in {1,2,...,6}{
     \draw[gray,line width=0.3](0,-\y) -- (11,-\y);
     }
     \draw (6.5,0) node[above] {\tiny $j$};
     \draw (10.5,0) node[above] {\tiny $j+m+1$};
     \draw (0,-3.5) node[left] {\tiny $i$};
     \draw[dashed] (6.5,0) -- (6.5,-3);
     \draw[dashed] (10.5,0) -- (10.5,-3);
     \draw[dashed] (0,-3.5) -- (6,-3.5);
     \draw (6.5,-3.5) node {$1$};
     \draw (7.5,-4.5) node {$1$};
     \draw (3.5,-4.5) node {$0$};
     \draw (4.5,-4.5) node {$0$};
     \draw (5.5,-4.5) node {$0$};
     \draw (6.5,-4.5) node {$0$};
     \draw (7.5,-3.5) node {$0$};
     \draw (8.5,-3.5) node {$0$};
     \draw (9.5,-3.5) node {$0$};
     \draw (10.5,-3.5) node {$0$};
     \draw[<->](8,-4.5)--(11,-4.5);
     \draw (9.5,-4.5)node[below]{\small$m$};
     \draw[<->](3,-3.5)--(6,-3.5);
     \draw (4.5,-2.5)node[below]{\small$m$};
     \draw[red,line width=1] (6,-3) -- ++(1,0) -- ++(0,-1) -- ++(-1,0) -- ++(0,1);
     \draw[red,line width=1] (7,-4) -- ++(1,0) -- ++(0,-1) -- ++(-1,0) -- ++(0,1);
     \end{scope}

    \end{tikzpicture}
    \caption{}\label{fig:10m1}
   \end{center}
  \end{figure}
 \end{center}

  \item 
	Suppose that  $S(i,j)=1$, that is, $\varphi^i(w)_j=1$. Then
	$S(i,j-k) = \varphi^i(i)_{j-k} = 0$ for  $k=1,2,\ldots,m$.
	If $S(i-1,j-1)=\varphi^i(w)_{j-m-1}=1$, then  $S(i,j-m-1)=0$. See 
	the left part of Figure $\ref{fig:4}$.
	Assume  $S(i-1,j-1)=0$ and $S(i-1,j-m-1)=0$, then  we have $S(i,j-1)=1$,
	which contradicts the assumption $S(i,j)=1$. See the right part of Figure $\ref{fig:4}$.

 \begin{center}
  \begin{figure}[H]
   \begin{center}
    \begin{tikzpicture}[scale=0.4]
     \draw(0,0) -- (11,0);
     \draw(0,0) -- (0,-7);
     \foreach \x in {1,2,...,10}{
     \draw[gray,line width=0.3](\x,0) -- (\x,-7);
     }
     \foreach \y in {1,2,...,6}{
     \draw[gray,line width=0.3](0,-\y) -- (11,-\y);
     }
     \draw (7.5,0) node[above] {\tiny $j$};
     \draw (3.5,0) node[above] {\tiny $j-m-1$};
     \draw (0,-4.5) node[left] {\tiny $i$};
     \draw[dashed] (7.5,0) -- (7.5,-3);
     \draw[dashed] (3.5,0) -- (3.5,-4);
     \draw[dashed] (0,-4.5) -- (3,-4.5);
     \draw (6.5,-3.5) node {$1$};
     \draw (7.5,-4.5) node {$1$};
     \draw (3.5,-4.5) node {$0$};
     \draw (4.5,-4.5) node {$0$};
     \draw (5.5,-4.5) node {$0$};
     \draw (6.5,-4.5) node {$0$};
     \draw (7.5,-3.5) node {$0$};
     \draw (8.5,-3.5) node {$0$};
     \draw (9.5,-3.5) node {$0$};
     \draw (10.5,-3.5) node {$0$};
     \draw[<->](3,-3.5)--(6,-3.5);
     \draw (4.5,-2.5)node[below]{\small$m$};
     \draw[red,line width=1] (6,-3) -- ++(1,0) -- ++(0,-1) -- ++(-1,0) -- ++(0,1);
     \draw[red,line width=1] (7,-4) -- ++(1,0) -- ++(0,-1) -- ++(-1,0) -- ++(0,1);
     \draw[line width=1] (3,-4) -- ++(1,0) -- ++(0,-1) -- ++(-1,0) -- ++(0,1);

     \begin{scope}[xshift=15cm]
           \draw(0,0) -- (11,0);
     \draw(0,0) -- (0,-7);
     \foreach \x in {1,2,...,10}{
     \draw[gray,line width=0.3](\x,0) -- (\x,-7);
     }
     \foreach \y in {1,2,...,6}{
     \draw[gray,line width=0.3](0,-\y) -- (11,-\y);
     }
     \draw (7.5,0) node[above] {\tiny $j$};
     \draw (3.5,0) node[above] {\tiny $j-m-1$};
     \draw (0,-4.5) node[left] {\tiny $i$};
     \draw[dashed] (7.5,0) -- (7.5,-3);
     \draw[dashed] (3.5,0) -- (3.5,-4);
     \draw[dashed] (0,-4.5) -- (3,-4.5);
     \draw (6.5,-3.5) node {$0$};
     \draw (7.5,-4.5) node {$1$};
     \draw (6.5,-4.5) node {$1$};
     \draw (3.5,-4.5) node {$0$};
     \draw (4.5,-4.5) node {$0$};
     \draw (5.5,-4.5) node {$0$};
     \draw (7.5,-3.5) node {$0$};
     \draw (8.5,-3.5) node {$0$};
     \draw (9.5,-3.5) node {$0$};
     \draw (10.5,-3.5) node {$0$};
     \draw[<->](3,-3.5)--(6,-3.5);
     \draw (4.5,-2.5)node[below]{\small$m$};
     \draw[line width=1] (6,-3) -- ++(1,0) -- ++(0,-1) -- ++(-1,0) -- ++(0,1);
     \draw[red,line width=1] (7,-4) -- ++(1,0) -- ++(0,-1) -- ++(-1,0) -- ++(0,1);
     \draw[red,line width=1] (6,-4) -- ++(1,0) -- ++(0,-1) -- ++(-1,0) -- ++(0,1);
     \draw[line width=1] (3,-4) -- ++(1,0) -- ++(0,-1) -- ++(-1,0) -- ++(0,1);
     \end{scope}
    \end{tikzpicture}
    \caption{}\label{fig:4}
   \end{center}
  \end{figure}
 \end{center}

 \end{enumerate}
\end{proof}

\begin{dfn}
 By Lemma $\ref{lem:snake}$, if $S(i,j)=1$, then 
a sequence $s = \left((r_0,j_0), (r_1,j_1),\ldots,(r_n,j_n)\right)$ containing
$(i,j)$ which has the following properties is uniquely determined.
\begin{enumerate}
 \item $j_0=0$, and $j_n=N$.
 \item $S(r_k,j_k) = 1$ for $k=0,1,\ldots,n$.
 \item 
       $(r_{k},j_k)-(r_{k-1},j_{k-1})\in\left\{(1,1),(0,m+1)\right\} \mbox{~~for~~}
       k=1,2,\ldots, n.$
\end{enumerate}
We call $s$ the {\em snake} containing $(i,j)$.
Since $j_k-j_{k-1}\in\{1,m+1\}$,
we obtain an $N-1$'s composition $(j_1-j_0)(j_2-j_1)\cdots(j_n-j_{n-1})$
whose parts are in $\{1,m+1\}$.
We call this composition the {\em snake composition} of $s$.

Let $c\in \{1,m+1\}^*$ be a snake composition. Then we can transform 
$c$ into a word in $\{0,1\}^*$ by replacing $m+1$ with $0$,
which we denote $\tilde{c}$.
\end{dfn}

\begin{exm}
 \label{exm:snake}
 Table $\ref{tab:snake}$ shows the orbit board of
 $10000000001000\in X_{14}$.
 A snake consists of positions of $1$'s which are marked by the circles.
 There are $9$ snakes in this orbit board.
 The snake compositions of these snakes are
 \[
  1    4    1    1    1    1    4,~ 
  1    1    1    1    4    4    1,~ 
  1    4    4    1    1    1    1,~ 
  4    1    1    1    1    4    1,~ 
  1    1    1    1    4    1    4,~ 
  1    4    1    4    1    1    1,~ 
  1    4    1    1    1    4    1,~ 
  1    1    1    4    1    4    1,~ 
  4    1    4    1    1    1    1.
 \]
 If we replace the digits $4$ in the above compositions
 with $0$, we obtain the $\rho$-orbit of $w=1011110$
 which we have already seen in Example $\ref{exm:first}$.
 We will explain this correspondence.
\end{exm}

\begin{lem}
 \label{lem:consecsnake}
 Suppose $S(i,j)=1$ and $S(i+2,j-d)=1$ with $m \leq d\leq 2m$ and $j\neq N-1$.
 Then, by Lemma $\ref{lem:snake}$, exactly one of $S(i+1,j+1)=1$ and $S(i,j+m+1)=1$ occurs,
 for each of which we have the following
 \begin{enumerate}
  \item If $S(i+1,j+1)=1$, then $S(i+3,j-d+1)=1$.
  \item If $S(i,j+m+1)=1$, then $S(i+2,j-d+m+1)=1$.
 \end{enumerate}
\end{lem}
\begin{proof}
 \begin{enumerate}
  \item 
	If $S(i+1,j+1)=1$, then we have
	\begin{equation}
	 \label{eq:conseczeros}
	  S(i+2,j-d+1)=S(i+2,j-d+2)=\cdots = S(i+2,j+1) = 0.
	\end{equation}
	In fact, since $S(i+2,j-d)=1$, we have 
	\begin{equation}
	 \label{eq:conseczeros1}
	  S(i+2,j-d+1)=S(i+2,j-d+2)=\cdots = S(i+2, j-d+m)=0,
	\end{equation}
	and since $S(i+1,j+1)=1$, we have
	\begin{equation}
	 \label{eq:conseczeros2}
	  S(i+2,j)=S(i+2,j-1)=\cdots = S(i+2,j+1-m)=0.
	\end{equation}
	Since $2m-d\geq 0$, $(j+1-m) - (j-d+m) = 1-(2m-d)  \leq 1$ and we obtain $(\ref{eq:conseczeros})$.
	Then Lemma $\ref{lem:snake}$ implies $S(i+3,j-d+1)=1$. (See Figure $\ref{fig:case1}$.)
 \begin{center}
  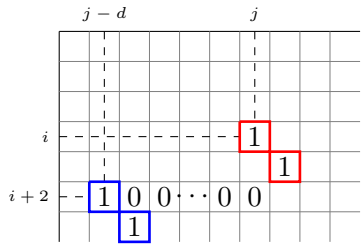
\begin{figure}[H]
   \begin{center}
    \begin{tikzpicture}[scale=0.4]
     \draw(0,0) -- (10,0);
     \draw(0,0) -- (0,-7);
     \foreach \x in {1,2,...,9}{
     \draw[gray,line width=0.3](\x,0) -- (\x,-7);
     }
     \foreach \y in {1,2,...,6}{
     \draw[gray,line width=0.3](0,-\y) -- (10,-\y);
     }
     \draw (6.5,0) node[above] {\tiny $j$};
     \draw (0,-3.5) node[left] {\tiny $i$};
     \draw (0,-5.5) node[left] {\tiny $i+2$};
     \draw[dashed] (6.5,0) -- (6.5,-3);
     \draw[dashed] (1.5,0) -- (1.5,-5);
     \draw[dashed] (0,-3.5) -- (6,-3.5);
     \draw[dashed] (0,-5.5) -- (1,-5.5);
     \draw (6.5,-3.5) node {$1$};
     \draw (7.5,-4.5) node {$1$};
     \draw (1.5,-5.5) node {$1$};
     \draw (2.5,-6.5) node {$1$};
     \draw (2.5,-5.5) node {$0$};
     \draw (3.5,-5.5) node {$0$};
     \draw (4.5,-5.5) node {$\cdots$};
     \draw (5.5,-5.5) node {$0$};
     \draw (6.5,-5.5) node {$0$};
     \draw (1.5,0) node[above] {\tiny $j-d$};
     \draw[red,line width=1] (6,-3) -- ++(1,0) -- ++(0,-1) -- ++(-1,0) -- ++(0,1);
     \draw[red,line width=1] (7,-4) -- ++(1,0) -- ++(0,-1) -- ++(-1,0) -- ++(0,1);
     \draw[blue,line width=1] (1,-5) -- ++(1,0) -- ++(0,-1) -- ++(-1,0) -- ++(0,1);
     \draw[blue,line width=1] (2,-6) -- ++(1,0) -- ++(0,-1) -- ++(-1,0) -- ++(0,1);
    \end{tikzpicture}
    \caption{Positions of 1's in Case 1}\label{fig:case1}
   \end{center}
  \end{figure}
 \end{center}
	
  \item If $S(i,j+m+1)=1$, then we have
	\[
	 S(i+1,j-k) = S(i+1, j-d+1) = \cdots S(i+1, j+m+1)=0,
	\]
	and therefore $S(i+2, j-d+m)=1$. (See Figure $\ref{fig:case2}$.)
 \end{enumerate}
\end{proof}

 \begin{center}
  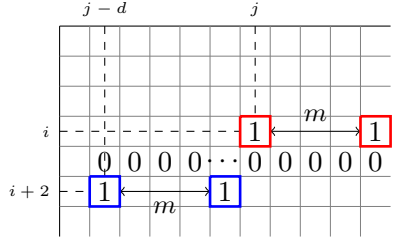
\begin{figure}[H]
   \begin{center}
    \begin{tikzpicture}[scale=0.4]
     \draw(0,0) -- (11,0);
     \draw(0,0) -- (0,-7);
     \foreach \x in {1,2,...,10}{
     \draw[gray,line width=0.3](\x,0) -- (\x,-7);
     }
     \foreach \y in {1,2,...,6}{
     \draw[gray,line width=0.3](0,-\y) -- (11,-\y);
     }
     \draw (6.5,0) node[above] {\tiny $j$};
     \draw (0,-3.5) node[left] {\tiny $i$};
     \draw (0,-5.5) node[left] {\tiny $i+2$};
     \draw[dashed] (6.5,0) -- (6.5,-3);
     \draw[dashed] (1.5,0) -- (1.5,-5);
     \draw[dashed] (0,-3.5) -- (6,-3.5);
     \draw[dashed] (0,-5.5) -- (1,-5.5);
     \draw (6.5,-3.5) node {$1$};
     \draw (10.5,-3.5) node {$1$};
     \draw (1.5,-5.5) node {$1$};
     \draw (1.5,-4.5) node {$0$};
     \draw (2.5,-4.5) node {$0$};
     \draw (3.5,-4.5) node {$0$};
     \draw (4.5,-4.5) node {$0$};
     \draw (5.5,-4.5) node {$\cdots$};
     \draw (6.5,-4.5) node {$0$};
     \draw (7.5,-4.5) node {$0$};
     \draw (8.5,-4.5) node {$0$};
     \draw (9.5,-4.5) node {$0$};
     \draw (10.5,-4.5) node {$0$};
     \draw (5.5,-5.5) node {$1$};
     \draw (1.5,0) node[above] {\tiny $j-d$};
     \draw[<->](7,-3.5)--(10,-3.5);
     \draw (8.5,-3.5)node[above]{\small$m$};
     \draw[<->](2,-5.5)--(5,-5.5);
     \draw (3.5,-5.5)node[below]{\small$m$};
     \draw[red,line width=1] (6,-3) -- ++(1,0) -- ++(0,-1) -- ++(-1,0) -- ++(0,1);
     \draw[red,line width=1] (10,-3) -- ++(1,0) -- ++(0,-1) -- ++(-1,0) -- ++(0,1);
     \draw[blue,line width=1] (1,-5) -- ++(1,0) -- ++(0,-1) -- ++(-1,0) -- ++(0,1);
     \draw[blue,line width=1] (5,-5) -- ++(1,0) -- ++(0,-1) -- ++(-1,0) -- ++(0,1);
    \end{tikzpicture}
    \caption{Positions of 1's in Case 2}\label{fig:case2}
   \end{center}
  \end{figure}
 \end{center}

\begin{lem}
 \label{lem:snaketail}
 Suppose $S(i,N-1)=1$.
 Then, there exists a unique $d$ such that
 $m \leq d\leq 2m$ and $S(i+2,N-1-d)=1$, and
 \begin{enumerate}
  \item \label{lem:snaketail1}
	If $d > m$, then  we have 
	\[
	S(i+2,N-d+m) = S(i+3,N-d+m+1) = \cdots = S(i+1+d-m, N-1) = 1.
	\]
  \item \label{lem:snaketail2}
	If $d = m$, then we have
	\[
	S(i+2,N-m-1) = S(i+3, N-m) = \cdots = S(i+2+m,N-1)=1.
	\]
 \end{enumerate}
\end{lem}
\begin{proof}
 The uniqueness of such $d$ is clear and
 $\ref{lem:snaketail1}$ and $\ref{lem:snaketail2}$ follows
 easily from the existence of such $d$ and Lemma $\ref{lem:snake}$
 by an argument parallel to Lemma $\ref{lem:consecsnake}$.
 Therefore we show the existence:
 It is clear that 
 \[
  S(i,N-2)=S(i,N-3)=\cdots=S(i,N-1-m)=0.
 \]
 If $S(i,N-2-m)=1$, then
 $S(i+1,N-2-2m)=S(i+1,N-1-2m)=S(i+1,N-2m)=\cdots = S(i+1,N-1)=0$.
 Therefore, there exists exactly one $d$ such that
 $S(i+2,N-1-d)=1$ and $m\leq d\leq 2m$.

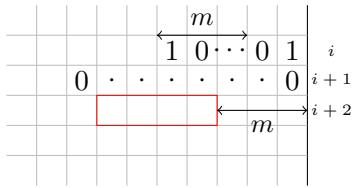
 \begin{center}
  \begin{figure}[H]
   \begin{center}
    \begin{tikzpicture}[scale=0.4]
     \draw(10,1) -- (10,-5);
     \foreach \y in {0,1,...,4}{
     \draw[lightgray] (10,-\y) -- (10-10,-\y);
     }
     \foreach \x in {1,...,9}{
     \draw[lightgray] (\x,1) -- (\x,-5);
     }
     \draw (9.5,-0.5) node {$1$};
     \draw (10-4-0.5,-0.5) node {$1$};
     \foreach \x in {3,...,8}{
     \draw (\x+0.5,-1.5) node {$\cdot$};
     }
     \draw (9+0.5,-1.5) node {$0$};
     \draw (2+0.5,-1.5) node {$0$};
     \draw[red] (3,-2) -- ++(4,0) -- ++(0,-1) -- ++(-4,0) -- cycle;
     \draw[<->](10,-2.5) -- (10-3,-2.5);
     \draw (10-1.5,-2.5) node[below] {\small $m$};
     \draw[<->](10-5,-.) -- (10-2,-.);
     \draw[](10-3.5,-.) node[above]{\small $m$};
     \draw (10+0.8,-0.5) node {\tiny $i$};
     \draw (10+0.8,-1.5) node {\tiny $i+1$};
     \draw (10+0.8,-2.5) node {\tiny $i+2$};
     \draw (8.5,-0.5) node {$0$};
     \draw (7.5,-0.5) node {$\cdots$};
     \draw (6.5,-0.5) node {$0$};
    \end{tikzpicture}
   \end{center}
  \end{figure}
 \end{center}

\end{proof}

\begin{thm}
 \label{thm:secondMain}
 Let $s$ be the snake containing $(i,0)$ in the $\varphi$-orbit board of $S\in X_N$. 
 Let $c$ be the snake composition of $s$ and 
 $i'$ be the least integer greater than $i$ for which $S(i',0)=1$.
 Assume $c$ starts with $k$-repetition of $1$, i.e., $\stackrel{k}{\overbrace{11\cdots{1}}}$.
 Then,
\begin{enumerate}
 \item\label{thm:kgtm}  If $k\geq m$, then $i'=i+m+2$.
 \item\label{thm:kltm} If $k< m$, then $i'=i+k+2$.
 \item\label{thm:tog2rot} Let $\lambda$ be the snake composition of the snake containing $(i,0)$ and let
       $\lambda'$ be the snake composition of the snake containing $(i',0)$.
       Then we have
      \begin{equation}
       \label{eq:tog2rot}
	\rho(\tilde{\lambda}) = \tilde{\lambda'},
      \end{equation}
       where $\tilde{\lambda}$ is the word in $\{0,1\}^*$ obtained from $\lambda$ by replacing
       $m+1$ with $0$.
\end{enumerate}

\end{thm}

\begin{proof}
 Without loss of generality, we can assume that $i=0$.

 If $k \geq m$, then 
 \[
  S(1,1) = S(2,2) = \cdots = S(m,m) = 1,
 \]
 and hence
 $S(i,j) = 0 \mbox{ for } i\leq m, 0\leq j < i$.
 Therefore
 \[
 S(m+1,0) = S(m+1,1) = \cdots  = S(m+1,m) = 0.
 \]
 This implies $S(m+2,0)=1$, which proves $\ref{thm:kgtm}$.

 If $k < m$, then 
 \[
  S(1,1)=S(2,2)=\cdots =S(k,k) =S(k,k+m+1)=1.
 \]
 and hence
 $S(k+1,0) = S(k+1,1) = \cdots = S(k+1,k+m+1) = 0$.
 This implies $S(k+2,0)=1$, which proves $\ref{thm:kltm}$.

 \bigskip
 We next prove $\ref{thm:tog2rot}$.
 Let $s=\left((0,j_0),(r_1,j_1),\ldots,(r_t,j_t)\right)$.
 If $k < m$, then 
 \[
  S(k,k)=S(k,k+m+1)=S(k+2,0)=1,
 \]
 and $(r_k,j_k)=(k,k)$ and $(r_{k+1},j_{k+1})=(k,k+m+1)$.
 In particular, the snake composition $c$ is expressed as
 \begin{eqnarray*}
  \lambda &= &(j_{1}-j_{0})(j_{2}-j_{1})\cdots(j_{t}-j_{t-1})\\
  &= &\stackrel{k}{\overbrace{11\cdots1}}(m+1)(j_{k+2}-j_{k+1})(j_{k+3}-j_{k+2})\cdots(j_{t}-j_{t-1}).
 \end{eqnarray*}
 Then, by Lemma $\ref{lem:consecsnake}$ and $\ref{lem:snaketail}$, the snake 
 composition $c'$ of the snake $s'$
 starting from $(k+2,0)$ is 
 \[
 (j_{k+2}-j_{k+1}) (j_{k+3}-j_{k+2}) \cdots (j_{t}-j_{t-1})(m+1)\stackrel{k}{\overbrace{11\cdots1}}.
 \]
 Therefore, $(\ref{eq:tog2rot})$ holds.
 If $k\geq m$, then 
 \[
  S(m,m)=S(m+2,0)=1.
 \]
 By Lemma $\ref{lem:consecsnake}$ and $\ref{lem:snaketail}$,
 the snake composition $\lambda'$ starting from $(m+2,0)$ is
 \[
 (j_{m+1}-j_{m}) (j_{m+2}-j_{m+1}) \cdots (j_{t}-j_{t-1})\stackrel{m}{\overbrace{11\cdots1}}.  
 \]
\end{proof}

\begin{thm}
\label{thm:thirdmain}
For every $S\in X_N$ and $i\in\{0,1,\ldots,n-1\}$
\[
 \sum_{t=0}^{q-1}\varphi^t(S)_i
 =
 \sum_{t=0}^{q-1}\varphi^t(S)_{N-1-i},
\]
where $q$ is the length of the $\varphi$-orbit of $S$
and $\varphi^t(S)_i$ is the $i$-th digit of the word $\varphi^t(S)$.
\end{thm}
\begin{proof}
 Let $s_0,s_1,\ldots, s_{p-1}$ be the snakes in the orbit board of $S$
 and $\lambda_0,\lambda_1,\ldots, \lambda_{p-1}$ be the corresponding snake compositions.
 Then, by Theorem $\ref{thm:secondMain}$, there exists a
 $\{0,1\}$-composition $w$ which satisfies
 \[
 \tilde{\lambda}_0 = w,~ \tilde{\lambda}_1 = \rho(w),~ \ldots, ~\tilde{\lambda}_{p-1}=\rho^{p-1}(w).
 \]
 Suppose that $\varphi^t(S)_i = 1$ and $(t,i)$ is contained in 
 a snake $s_k$.
 Let the snake composition of $s_k$ be $\lambda_k=u_0u_1\cdots u_{n-1}$.
 Then, we have some index $j$ such that 
 $i=u_{0}+u_{1}+\cdots+u_{j-1}$,
 which is equal to 
 \[
 (m+1)j - m\left(
  \rho^k(w)_0 +
  \rho^k(w)_1 +
 \cdots +
 \rho^k(w)_{j-1}
 \right).
 \]
 This implies
 \[
  \rho^k(w)_0 +
  \rho^k(w)_1 +
 \cdots +
 \rho^k(w)_{j-1}
 =
 j - \frac{i-j}{m}.
 \]
 Therefore
 \[
 \sum_{t=0}^{q-1}\varphi^t(S)_i
 =
 \sum_{j=0}^{|w|-1} \#\left\{k\,\middle|\,\sum_{\nu =0}^{j-1}\rho^k(w)_\nu = j - \frac{i-j}{m}\right\}
 =
 \sum_{j=0}^{|w|-1} \nu_{L^{(j)}}\left(j - \frac{i-j}{m} \right).
 \]
 In the same manner, we can show
 \[
  \sum_{t=0}^{q-1}\varphi^t(S)_{N-1-i}
 =
 \sum_{j=0}^{|w|-1} \nu_{R^{(j)}}\left(j - \frac{i-j}{m} \right).
 \]
 By Theorem $\ref{thm:main}$, we are done.
\end{proof}

\begin{exm}
 We can compute the bottom line,
 $9, 7, 3, 3, 4, 6,4,\ldots, 9$ of the Table $\ref{tab:snake}$
 from the table in Figure $\ref{fig:freqtab}$
 as follows,
\[
 \begin{array}{llll}
  9 = \nu_{L^{(0)}}\left(0\right) + \nu_{L^{(3)}}\left(4\right),& 
   7 = \nu_{L^{(1)}}\left(1\right) + \nu_{L^{(4)}}\left(5\right),&
   3 = \nu_{L^{(2)}}\left(2\right),& 
   3 = \nu_{L^{(3)}}\left(3\right), \\
  4 = \nu_{L^{(1)}}\left(0\right) + \nu_{L^{(4)}}\left(4\right),&
   6 = \nu_{L^{(2)}}\left(1\right) + \nu_{L^{(5)}}\left(5\right),&
   4 = \nu_{L^{(3)}}\left(2\right) + \nu_{L^{(6)}}\left(6\right),
 \end{array}
\]
and so forth.
\end{exm}

\begin{exm}
 When $m=3$ the set $X_{14}$  has
 cardinality $131$.
 By using Proposition $\ref{prop:orbitsize}$, we can  efficiently decompose
 $X_N$ into $\varphi$-orbits.

\begin{table}[H]
 \begin{center}
  \begin{tabular}{|l|l|l|l|l|}
   \hline
   $S$ & $\tilde{c}$ & ${\rm bqw}(\tilde{c})$ & ${\rm rw}(\tilde{c})$ & period \\
   \hline
	   10000100001000 & 1111111111111 & 1111 & 1 & 5\\
   10000001000010 & 1111111101 & 110 & 21 & 27\\
   10000010000100 & 1111111110 & 1110 & 00 & 17\\
   10000001000100 & 1101101 & 00 & 221 & 11\\
   10000010000001 & 1111010 & 100 & 110 & 31\\
   10000010000010 & 1111100 & 100 & 200 & 31\\
   10000000000000 & 1000 & 000 & 1000 & 9\\
   \hline
  \end{tabular}
  \caption{$\varphi$-orbits in $X_{14}$ with $m=3$.
  $S$ denotes an element of $X_{14}$ and $c$ is the
  first snake composition appearing in
  the orbit board of $S$.}
 \end{center}
\end{table}
\end{exm}

\section{Concluding remarks}

We have not succeeded in fully generalizing the results on
the number of orbits and the orbit sizes by
Joseph and Roby \cite{JR}. 
Our generalization of \cite{JR} can be considered as the toggle dynamical
systems on {\em more} independent sets.
It seems that results in this paper can be further generalized.
For example, we can also consider {\em less} independent sets:
Let $Z_N$ be the set defined by
\[
 Z_N = 
 \left\{
 z=z_0 z_1\cdots z_{N-1} \in \{0,1\}^N
 \,\middle|\,
 z_i + z_{i+1} + \cdots + z_{i+m} \leq m
 \mbox{ for }
 i = 0,1,\ldots, N-m
 \right\}.
\]
In other words, $Z_N$ is the set 
consisting of the words of length $N$ which do not contain 
$m+1$ consecutive $1$s as its subword.
Therefore $X_N = Z_N$ when $m=1$.
Numerical experiments suggests that the toggle dynamical system
on $Z_N$ has the same symmetric property as $X_N$.
However, we have not succeeded in finding the objects corresponding to  the {\em snakes}
on the orbit board of $z\in Z_N$.

\section*{Acknowledgment}
The authors thank Tom Roby for helpful discussions and suggestions.
They deeply thank the anonymous reviewers  who carefully read
through the manuscript and gave many helpful suggestions.


\end{document}